\documentclass[
11pt,                          
english                        
]{article}

\usepackage[english]{babel}    
\usepackage{amsmath}           
\usepackage[utf8]{inputenc}    
\usepackage[T1]{fontenc}       
\usepackage{exscale}           
\usepackage[sort]{cite}        
\usepackage[a4paper]{geometry} 
\usepackage{xspace}            
\usepackage{amssymb}           
\usepackage{bbm}               
\usepackage{mathtools}         
\usepackage[amsmath,thmmarks,hyperref]{ntheorem} 
\usepackage{mathrsfs}          
\usepackage{stmaryrd}          
\usepackage{paralist}          
\usepackage[nottoc]{tocbibind} 
\usepackage[backref=page,      
           pdfpagelabels       
          ]{hyperref}         

\author{
  \textbf{Chiara Esposito}\thanks{\texttt{chiara.esposito@mathematik.uni-wuerzburg.de}},\\
    Institut für Mathematik \\
  Lehrstuhl für Mathematik X \\
  Universität Würzburg \\
  Campus Hubland Nord \\
  Emil-Fischer-Straße 31 \\
  97074 Würzburg \\
  Germany \\[0.3cm]
  \textbf{Jonas Schnitzer}\thanks{\texttt{jschnitzer@unisa.it}},\\
    Dipartimento di Matematica\\
  Universit\`a degli Studi di Salerno\\
  via Giovanni Paolo II, 123\\
  84084 Fisciano (SA)\\
  Italy
    \\[0.3cm]
  \textbf{Stefan Waldmann}\thanks{\texttt{stefan.waldmann@mathematik.uni-wuerzburg.de}}\\
  Institut für Mathematik \\
  Lehrstuhl für Mathematik X \\
  Universität Würzburg \\
  Campus Hubland Nord \\
  Emil-Fischer-Straße 31 \\
  97074 Würzburg \\
  Germany
}

\renewcommand{\mathbb}[1]{\mathbbm{#1}}

\newcommand{\refitem}[1] {\textit{\ref{#1}.)}}

\numberwithin{equation}{section}

\allowdisplaybreaks

\let\originalleft\left
\let\originalright\right
\renewcommand{\left}{\mathopen{}\mathclose\bgroup\originalleft}
\renewcommand{\right}{\aftergroup\egroup\originalright}

\theoremheaderfont{\normalfont\bfseries}
\theorembodyfont{\itshape}
\newtheorem{lemma}{Lemma}[section]
\newtheorem{proposition}[lemma]{Proposition}
\newtheorem{theorem}[lemma]{Theorem}
\newtheorem{corollary}[lemma]{Corollary}
\newtheorem{definition}[lemma]{Definition}

\theorembodyfont{\rmfamily}

\newtheorem{example}[lemma]{Example}
\newtheorem{remark}[lemma]{Remark}

\theoremheaderfont{\scshape}
\theorembodyfont{\normalfont}
\theoremstyle{nonumberplain}
\theoremseparator{:}
\theoremsymbol{\hbox{$\boxempty$}}
\newtheorem{proof}{Proof}

 \newenvironment{lemmalist}{\begin{compactenum}[\itshape i.)]}{\end{compactenum}}
 \newenvironment{theoremlist}{\begin{compactenum}[\itshape i.)]}{\end{compactenum}}
 \newenvironment{propositionlist}{\begin{compactenum}[\itshape i.)]}{\end{compactenum}}

\newcommand\ot[2]{\stackrel{\mathclap{#1}}{#2}}
\newcommand{\uea}[1]       {\mathscr{U}({#1})}
\newcommand{\ueac}[1]      {\mathscr{U}_{\ring{C}}({#1})}
\newcommand{\twist}[1]     {\mathcal{#1}}
\newcommand{\Deg}          {\operatorname{Deg}}
\newcommand{\circs}        {\mathbin{\circ_{\pi}}}
\newcommand{\circweyl}     {\mathbin{\circ_{\scriptscriptstyle{\mathrm{Weyl}}}}}
\newcommand{\circwick}     {\mathbin{\circ_{\scriptscriptstyle{\mathrm{Wick}}}}}
\newcommand{\starwick}     {\mathbin{\star_{\scriptscriptstyle{\mathrm{Wick}}}}}
\newcommand{\tauwick}      {\tau_{\scriptscriptstyle{\mathrm{Wick}}}}
\newcommand{\ms}           {m_{\pi}}
\newcommand{\msA}          {m^\algebra{A}_{\pi}}

\newcommand{\lift}         {\mathrm{\scriptscriptstyle Lift}}
\newcommand{\CE}           {\mathrm{\scriptscriptstyle CE}}
\newcommand{\FedosovD}     {\mathscr{D}_{\mathrm{F}}}
\newcommand{\FedosovA}     {\mathscr{D}_{\algebra{A}}}
\newcommand{\FedosovTU}    {\mathscr{D}}

\newcommand{\lie}[1]          {\mathfrak{#1}}

\newcommand{\tensor}[1][{}]           {\mathbin{\otimes_{\scriptscriptstyle{#1}}}}

\newcommand{\algebra}[1]      {\mathscr{#1}}

\newcommand{\acts}            {\mathbin{\triangleright}}

\newcommand{\Cinfty}         {\mathscr{C}^\infty}

\DeclarePairedDelimiter{\Schouten}{\llbracket}{\rrbracket}

\newcommand{\argument}       {\,\cdot\,}

\newcommand{\Anti}                    {\Lambda}

\newcommand{\id}             {\mathsf{id}}

\newcommand{\ring}[1]         {\mathsf{#1}}
\newcommand{\Tensor}                  {\mathrm{T}}
\newcommand{\Sym}                     {\mathrm{S}}
\newcommand{\I}              {\mathrm{i}}
\newcommand{\E}              {\mathrm{e}}
\DeclareMathOperator{\degs}           {\mathrm{deg}_{\scriptscriptstyle\mathrm s}}
\DeclareMathOperator{\dega}           {\mathrm{deg}_{\scriptscriptstyle\mathrm a}}
\DeclareMathOperator{\insa}           {\mathrm{i}_{\scriptscriptstyle\mathrm a}}
\DeclareMathOperator{\inss}           {\mathrm{i}_{\scriptscriptstyle\mathrm s}}
\DeclareMathOperator{\ad}     {\mathrm{ad}}
\newcommand{\End}            {\operatorname{\mathsf{End}}}
\newcommand{\at}[1]          {\big|_{#1}}

\newcommand{\cc}[1]          {\overline{{#1}}}
\newcommand{\Mat}            {\mathrm{M}}
\newcommand{\Trans}          {{\mathrm{\scriptscriptstyle{T}}}}
\DeclareMathOperator{\tr}    {\mathsf{tr}}
\DeclareMathOperator{\diag}  {\mathrm{diag}}

\title{A Universal Construction of Universal Deformation Formulas,
  Drinfel'd Twists and their Positivity}

\begin{document}

%
%

\maketitle

%
%
\begin{abstract}
    In this paper we provide an explicit construction of star products
    on $\uea{\lie{g}}$-module algebras by using the Fedosov
    approach. This construction allows us to give a constructive proof
    to Drinfel'd theorem and to obtain a concrete formula for
    Drinfel'd twist.  We prove that the equivalence classes of twists
    are in one-to-one correspondence with the second
    Chevalley-Eilenberg cohomology of the Lie algebra $\lie{g}$.
    Finally, we show that for Lie algebras with Kähler structure we
    obtain a strongly positive universal deformation of $^*$-algebras
    by using a Wick-type deformation. This results in a positive
    Drinfel'd twist.
\end{abstract}

%
%

\tableofcontents

\newpage

\section{Introduction}

The concept of deformation quantization has been defined by Bayen,
Flato, Fronsdal, Lichnerowicz and Sternheimer in
\cite{bayen.et.al:1978a} based on Gerstenhaber's theory of associative
deformations of algebra \cite{gerstenhaber:1964a}.  A formal star
product on a symplectic (or Poisson) manifold $M$ is defined as a
formal associative deformation $\star$ of the algebra of smooth
functions $\Cinfty (M)$ on $M$. The existence as well as the
classification of star products has been studied in many different
settings, e.g in \cite{dewilde.lecomte:1983b, fedosov:1986a,
  fedosov:1994a, fedosov:1996a, kontsevich:2003a, nest.tsygan:1995a,
  bertelson.cahen.gutt:1997a}, see also the textbooks
\cite{esposito:2015a, waldmann:2007a} for more details in deformation
quantization. Quite parallel to this, Drinfel'd introduced the notion
of quantum groups and started the deformation of Hopf algebra, see
e.g. the textbooks \cite{kassel:1995a, chari.pressley:1994a,
  etingof.schiffmann:1998a} for a detailed discussion.

It turned out that under certain circumstances one can give simple and
fairly explicit formulas for associative deformations of algebras:
whenever a Lie algebra $\lie{g}$ acts on an associative algebra
$\algebra{A}$ by derivations, the choice of a \emph{formal Drinfel'd
  twist} $\twist{F} \in (\uea{\lie{g}} \tensor \uea{\lie{g}})[[t]]$
allows to deform $\algebra{A}$ by means of a \emph{universal
  deformation formula}
\begin{equation}
    \label{eq:TheUDF}
    a \star_{\twist{F}} b
    =
    \mu_{\algebra A} (\twist{F}\acts (a \tensor b))
\end{equation}
for $a, b \in \algebra{A}[[t]]$. Here
$\mu_{\algebra{A}}\colon \algebra{A} \tensor \algebra{A}
\longrightarrow \algebra{A}$
is the algebra multiplication and $\acts$ is the action of $\lie{g}$
extended to the universal enveloping algebra $\uea{\lie{g}}$ and then
to $\uea{\lie{g}} \tensor \uea{\lie{g}}$ acting on
$\algebra{A} \tensor \algebra{A}$. Finally, all operations are
extended $\ring{R}[[t]]$-multilinearly to formal power series.  Recall
that a formal Drinfel'd twist \cite{drinfeld:1983a, drinfeld:1988a} is
an invertible element
$\twist{F} \in (\uea{\lie{g}} \tensor \uea{\lie{g}})[[t]]$ satisfying
\begin{gather}
    \label{eq:TwistConditionI}
    (\Delta \tensor \id)(\twist{F})(\twist{F}\tensor 1)
    =
    (\id \tensor \Delta)(\twist{F})(1 \tensor \twist{F}),
    \\
    \label{eq:TwistConditionII}
    (\epsilon \tensor 1)\twist{F}
    =
    1
    =
    (1\tensor \epsilon)\twist{F} \\
    \shortintertext{and}
    \label{eq:TwistConditionIII}
    \twist{F} = 1 \tensor 1 + \mathcal{O}(t).
\end{gather}
The properties of a twist are now easily seen to guarantee that
\eqref{eq:TheUDF} is indeed an associative deformation.

Yielding the explicit formula for the deformation universally in the
algebra $\algebra{A}$, Drinfel'd twists are considered to be of great
importance in deformation theory in general, and in fact, used at many
different places. We just mention a few recent developments, certainly
not exhaustive: Giaquinto and Zhang studied the relevance of universal
deformation formulas like \eqref{eq:TheUDF} in great detail in the
seminal paper \cite{giaquinto.zhang:1998a}. Bieliavsky and Gayral
\cite{bieliavsky.gayral:2015a} used universal deformation formulas
also in a non-formal setting by replacing the notion of a Drinfel'd
twist with a certain integral kernel. This sophisticated construction
leads to a wealth of new strict deformations having the above formal
deformations as asymptotic expansions. But also beyond pure
mathematics the universal deformation formulas found applications
e.g. in the construction of quantum field theories on noncommutative
spacetimes, see e.g. \cite{aschieri.schenkel:2014a}.

In characteristic zero, there is one fundamental example of a
Drinfel'd twist in the case of an abelian Lie algebra $\lie{g}$. Here
one chooses any bivector $\pi \in \lie{g} \tensor \lie{g}$ and
considers the formal exponential
\begin{equation}
    \label{eq:WeylTwist}
    \twist{F}_{\textrm{Weyl-Moyal}}
    =
    \exp(t\pi),
\end{equation}
viewed as element in $(\uea{\lie{g}} \tensor \uea{\lie{g}})[[t]]$. An
easy verification shows that this is indeed a twist. The corresponding
universal deformation formula goes back at least till
\cite[Thm.~8]{gerstenhaber:1968a} under the name of \emph{deformation
  by commuting derivations}. In deformation quantization the
corresponding star product is the famous Weyl-Moyal star product if
one takes $\pi$ to be antisymmetric.

While this is an important example, it is not at all easy to find
explicit formulas for twists in the general non-abelian case. A
starting point is the observation, that the antisymmetric part of the
first order of a twist, $\twist{F}_1 - \mathsf{T}(\twist{F}_1)$, where
$\mathsf{T}$ is the usual flip isomorphism, is first an element in
$\Anti^2\lie{g}$ instead of $\Anti^2 \uea{\lie{g}}$, and, second, a
\emph{classical $r$-matrix}. This raises the question whether one can
go the opposite direction of a quantization: does every classical
$r$-matrix $r \in \Anti^2\lie{g}$ on a Lie algebra $\lie{g}$ arise as
the first order term of a formal Drinfel'd twist? It is now a
celebrated theorem of Drinfel'd \cite[Thm.~6]{drinfeld:1983a} that
this is true.

But even more can be said: given a twist $\twist{F}$ one can construct
a new twist by conjugating with an invertible element
$S \in \uea{\lie{g}}[[t]]$ starting with $S = 1 + \mathcal{O}(t)$ and
satisfying $\epsilon(S) = 1$. More precisely,
\begin{equation}
    \label{eq:NewTwist}
    \twist{F}' = \Delta(S)^{-1} \twist{F} (S \tensor S)
\end{equation}
turns out to be again a twist. In fact, this defines an equivalence
relation on the set of twists, preserving the semi-classical limit,
i.e. the induced $r$-matrix. In the spirit of Kontsevich's formality
theorem, and in fact building on its techniques, Halbout showed that
the equivalence classes of twists quantizing a given classical
$r$-matrix are in bijection to the equivalence classes of formal
deformations of the $r$-matrix in the sense of $r$-matrices
\cite{halbout:2006a}. In fact, this follows from Halbout's more
profound result on formality for general Lie bialgebras, the
quantization of $r$-matrices into twists is just a special case
thereof.  His theorem holds in a purely algebraic setting (in
characteristic zero) but relies heavily on the fairly inexplicit
formality theorems of Kontsevich and Tamarkin \cite{tamarkin:1998b}
which in turn require a rational Drinfel'd associator.

On the other hand, there is a simpler approach to the existence of
twists in the case of real Lie algebras: in seminal work of Drinfel'd
\cite{drinfeld:1983a} he showed that a twist is essentially the same
as a left $G$-invariant star product on a Lie group $G$ with Lie
algebra $\lie{g}$, by identifying the $G$-invariant bidifferential
operators on $G$ with elements in $\uea{\lie{g}} \tensor
\uea{\lie{g}}$. The associativity of the star product gives then
immediately the properties necessary for a twist and vice
versa. Moreover, an $r$-matrix is nothing else as a left $G$-invariant
Poisson structure, see \cite[Thm.~1]{drinfeld:1983a}. In this paper,
Drinfel'd also gives an existence proof of such $G$-invariant star
products and therefore of twists, see
\cite[Thm.~6]{drinfeld:1983a}. His argument uses the canonical star
product on the dual of a central extension of the Lie algebra by the
cocycle defined by the (inverse of the) $r$-matrix, suitably pulled
back to the Lie group, see also
Remark~\ref{remark:DrinfeldConstruction} for further details.

The equivalence of twists translates into the usual $G$-invariant
equivalence of star products as discussed in
\cite{bertelson.bieliavsky.gutt:1998a}.  Hence one can use the
existence (and classification) theorems for invariant star products to
yield the corresponding theorems for twists
\cite{bieliavsky:2016a}. This is also the point of view taken by
Dolgushev et al. in \cite{dolgushev.isaev.lyakhovich.sharapov:2002a},
where the star product is constructed in a way inspired by Fedosov's
construction of star products on symplectic manifolds.

A significant simplification concerning the existence comes from the
observation that for every $r$-matrix $r \in \Anti^2 \lie{g}$ there is
a Lie subalgebra of $\lie{g}$, namely
\begin{equation}
    \label{eq:ESSubalgebra}
    \lie{g}_r
    =
    \left\{
        (\alpha \tensor \id)(r)
        \; \big| \;
        \alpha \in \lie{g}^*
    \right\},
\end{equation}
such that $r \in \Anti^2 \lie{g}_r$ and $r$ becomes
\emph{non-degenerate} as an $r$-matrix on this Lie subalgebra
\cite[Prop.~3.2-3.3]{etingof.schiffmann:1998a}. Thus it will always be
sufficient to consider non-degenerate classical $r$-matrices when
interested in the existence of twists. For the classification this is
of course not true since a possibly degenerate $r$-matrix might be
deformed into a non-degenerate one only in higher orders: here one
needs Halbout's results for possibly degenerate $r$-matrices. However,
starting with a non-degenerate $r$-matrix, one will have a much
simpler classification scheme as well.

The aim of this paper is now twofold: On the one hand, we want to give
a direct construction to obtain the universal deformation formulas for
algebras acted upon by a Lie algebra with non-degenerate
$r$-matrix. This will be obtained in a purely algebraic fashion for
sufficiently nice Lie algebras and algebras over a commutative ring
$\ring{R}$ containing the rationals. Our approach is based on a
certain adaption of the Fedosov construction of symplectic star
products, which is in some sense closer to the original Fedosov
construction compared to the approach of
\cite{dolgushev.isaev.lyakhovich.sharapov:2002a} but yet completely
algebraic. More precisely, the construction will not involve a twist
at all but just the classical $r$-matrix. Moreover, it will be
important to note that we can allow for a non-trivial symmetric part
of the $r$-matrix, provided a certain technical condition on it is
satisfied. This will produce deformations with more specific features:
as in usual deformation quantization one is not only interested in the
Weyl-Moyal like star products, but certain geometric circumstances
require more particular star products like Wick-type star products on
Kähler manifolds \cite{karabegov:2013a, karabegov:1996a,
  bordemann.waldmann:1997a} or standard-ordered star products on
cotangent bundles \cite{bordemann.neumaier.waldmann:1998a,
  bordemann.neumaier.pflaum.waldmann:2003a}.

On the other hand, we give an alternative construction of Drinfel'd
twists, again in the purely algebraic setting, based on the above
correspondence to star products but avoiding the techniques from
differential geometry completely in order to be able to work over a
general field of characteristic zero. We also obtain a classification
of the above restricted situation where the $r$-matrix is
non-degenerate.

In fact, both questions turn out to be intimately linked since
applying our universal deformation formula to the tensor algebra of
$\uea{\lie{g}}$ will yield a deformation of the tensor product which
easily allows to construct the twist. This is in so far remarkable
that the tensor algebra is of course rigid, the deformation is
equivalent to the undeformed tensor product, but the deformation is
not the identity, allowing therefore to consider nontrivial products
of elements in $\Tensor^\bullet(\uea{\lie{g}})$.

We show that the universal deformation formula we construct in fact
coincides with \eqref{eq:TheUDF} for the twist we construct. However,
it is important to note that the detour via the twist is not needed to
obtain the universal deformation of an associative algebra.

Finally, we add the notion of positivity: this seems to be new in the
whole discussion of Drinfel'd twists and universal deformation
formulas so far. To this end we consider now an ordered ring
$\ring{R}$ containing $\mathbb{Q}$ and its complex version $\ring{C} =
\ring{R}(\I)$ with $\I^2 = -1$, and $^*$-algebras over $\ring{C}$ with
a $^*$-action of the Lie algebra $\lie{g}$, which is assumed to be a
Lie algebra over $\ring{R}$ admitting a Kähler structure. Together
with the non-degenerate $r$-matrix we can define a Wick-type universal
deformation which we show to be \emph{strongly positive}: every
undeformed positive linear functional stays positive also for the
deformation. Applied to the twist we conclude that the Wick-type twist
is a convex series of positive elements.

The paper is organized as follows. In Section~\ref{sec:FedosovSetUp}
we explain the elements of the (much more general) Fedosov
construction which we will
need. Section~\ref{sec:UniversalDeformationFormula} contains the
construction of the universal deformation formula. Here not only the
deformation formula will be universal for all algebras $\algebra{A}$
but also the construction itself will be universal for all Lie
algebras $\lie{g}$. In
Section~\ref{sec:UniversalDeformationFormulaTwist} we construct the
Drinfel'd twist while Section~\ref{sec:Classification} contains the
classification in the non-degenerate case. Finally,
Section~\ref{sec:HermitianCPDeformations} discusses the positivity of
the Wick-type universal deformation formula. In two appendices we
collect some more technical arguments and proofs.  The results of this
paper are partially based on the master thesis \cite{schnitzer:2016a}.

For symplectic manifolds with suitable polarizations one can define
various types of star products with separation of variables
\cite{karabegov:1996a, karabegov:2013a, bordemann.waldmann:1997a,
  bordemann.neumaier.waldmann:1999a, donin:2003a,
  bordemann.neumaier.waldmann:1998a,
  bordemann.neumaier.pflaum.waldmann:2003a} which have specific
properties adapted to the polarization. The general way to construct
(and classify) them is to modify the Fedosov construction by adding
suitable symmetric terms to the fiberwise symplectic Poisson
tensor. We have outlined that this can be done for twists as well in
the Kähler case, but there remain many interesting situations. In
particular a more cotangent-bundle like polarization might be
useful. We plan to come back to these questions in a future project.

\noindent
\textbf{Acknowledgements:} We would like to thank Pierre Bieliavsky,
Kenny De Commer, Alexander Karabegov, and Thorsten Reichert for the
discussions and useful suggestions. Moreover, we would like to thank
the referee for many useful comments and remarks.

%
%

\section{The Fedosov Set-Up}
\label{sec:FedosovSetUp}

In the following we present the Fedosov approach in the particular
case of a Lie algebra $\lie{g}$ with a non-degenerate $r$-matrix $r$.
We follow the presentation of Fedosov approach given in
\cite{waldmann:2007a} but replacing differential geometric concepts by
algebraic version in order to be able to treat not only the real case.
The setting for this work will be to assume that $\lie{g}$ is a Lie
algebra over a commutative ring $\ring{R}$ containing the rationals
$\mathbb{Q} \subseteq \ring{R}$ such that $\lie{g}$ is a
finite-dimensional free module.

We denote by $\{e_1, \ldots, e_n\}$ a basis of $\lie{g}$ and by
$\{e^1, \ldots, e^n\}$ its dual basis of $\lie{g}^*$.  We also assume
the $r$-matrix $r \in \Anti^2\lie{g}$ to be non-degenerate in the
strong sense from the beginning, since, at least in the case of
$\ring{R}$ being a field, we can replace $\lie{g}$ by $\lie{g}_r$ from
\eqref{eq:ESSubalgebra} if necessary. Hence $r$ induces the
\emph{musical isomorphism}
\begin{equation}
    \label{eq:MusicalIso}
    \sharp\colon
    \lie{g}^* \longrightarrow \lie{g}
\end{equation}
by paring with $r$, the inverse of which we denote by $\flat$ as
usual.  Then the defining property of an $r$-matrix is
$\Schouten{r, r} = 0$, where $\Schouten{\argument, \argument}$ is the
unique extension of the Lie bracket to $\Anti^\bullet \lie{g}$ turning
the Grassmann algebra into a Gerstenhaber algebra. Since we assume $r$
to be (strongly) non-degenerate have the inverse
$\omega \in \Anti^2 \lie{g}^*$ of $r$ and $\Schouten{r, r} = 0$
becomes equivalent to the linear condition $\delta_\CE\omega = 0$,
where $\delta_\CE$ is the usual Chevalley-Eilenberg
differential. Moreover, the musical isomorphisms intertwine
$\delta_\CE$ on $\Anti^\bullet \lie{g}^*$ with the differential
$\Schouten{r, \argument}$ on $\Anti^\bullet \lie{g}$. We refer to
$\omega$ as the induced symplectic form.
\begin{remark}
    \label{remark:WhyRings}%
    For the Lie algebra $\lie{g}$ there seems to be little gain in
    allowing a ring $\ring{R}$ instead of a field $\mathbb{k}$ of
    characteristic zero, as we have to require $\lie{g}$ to be a free
    module and \eqref{eq:MusicalIso} to be an isomorphism. However,
    for the algebras which we would like to deform there will be no
    such restrictions later on. Hence allowing for algebras over rings
    in the beginning seems to be the cleaner way to do it, since after
    the deformation we will arrive at an algebra over a ring, namely
    $\ring{R}[[t]]$ anyway.
\end{remark}
\begin{definition}[Formal Weyl algebra]
    The algebra $\left(\prod_{k=0}^\infty \Sym^k\lie{g}^* \tensor
        \Anti^\bullet \lie{g}^*\right)[[t]]$ is called the formal Weyl
    algebra where the product $\mu$ is defined by
    \begin{equation}
        \label{eq:UndeformedProduct}
        (f \tensor \alpha) \cdot (g \tensor \beta)
        =
        \mu(f \tensor \alpha, g \tensor \beta)
        =
        f\vee g\tensor\alpha\wedge\beta.
    \end{equation}
    for any factorizing tensors $f \tensor \alpha, g \tensor \beta \in
    \mathcal{W} \tensor \Anti^\bullet$ and extended
    $\ring{R}[[t]]$-bilinearly. We write $\mathcal{W} =
    \prod_{k=0}^\infty \Sym^k \lie{g}^* [[t]]$ and $\Anti^\bullet =
    \Anti^\bullet \lie{g}^*[[t]]$.
\end{definition}
Since $\lie{g}$ is assumed to be finite-dimensional we have
\begin{equation}
    \label{eq:WtensorLambdaIsWtensorLambda}
    \mathcal{W} \tensor \Lambda^\bullet
    =
    \left(
        \prod_{k=0}^\infty \Sym^k\lie{g}^* \tensor
        \Anti^\bullet \lie{g}^*
    \right)[[t]].
\end{equation}

Since we will deform this product $\mu$ we shall refer to $\mu$ also
as the \emph{undeformed product} of
$\mathcal{W} \tensor \Anti^\bullet$. It is clear that $\mu$ is
associative and graded commutative with respect to the antisymmetric
degree. In order to handle this and various other degrees, it is
useful to introduce the following degree maps
\begin{equation}
    \label{eq:DegreeMaps}
    \degs, \dega, \deg_t\colon
    \mathcal{W} \tensor \Anti^\bullet
    \longrightarrow
    \mathcal{W} \tensor \Anti^\bullet,
\end{equation}
defined by the conditions
\begin{equation}
    \label{eq:DegsDega}
    \degs(f \tensor \alpha)
    =
    k f \tensor \alpha
    \quad
    \textrm{and}
    \quad
    \dega(f \tensor\alpha)
    =
    \ell f \tensor \alpha
\end{equation}
for $f \in \Sym^k\lie{g}^*$ and $\alpha \in \Anti^\ell \lie{g}^*$. We
extend these maps to formal power series by $\ring{R}[[t]]$-linearity.
Then we can define the degree map $\deg_t$ by
\begin{equation}
    \label{eq:Degt}
    \deg_t
    =
    t\frac{\partial}{\partial t},
\end{equation}
which is, however, not $\ring{R}[[t]]$-linear.  Finally, the
\emph{total degree} is defined by
\begin{equation}
    \label{eq:TotalDegree}
    \Deg
    =
    \degs + 2\deg_t.
\end{equation}
It will be important that all these maps are derivations of the
undeformed product $\mu$ of $\mathcal{W} \tensor \Anti^\bullet$.  We
denote by
\begin{equation}
    \label{eq:FiltrationDeg}
    \mathcal{W}_k \tensor \Anti^\bullet
    =
    \bigcup_{r \ge k}
    \left\{
        a \in \mathcal{W} \tensor \Anti^\bullet
        \; \big| \;
        \Deg a = r a
    \right\}
\end{equation}
the subspace of elements which have total degree bigger or equal to
$+k$. This endows $\mathcal{W} \tensor \Anti^\bullet$ with a complete
filtration, a fact which we shall frequently use in the
sequel. Moreover, the filtration is compatible with the undeformed
product \eqref{eq:UndeformedProduct} in the sense that
\begin{equation}
    \label{eq:FiltrationProduct}
    ab \in
    \mathcal{W}_{k + \ell} \tensor \Anti^\bullet
    \quad
    \textrm{for}
    \quad
    a \in \mathcal{W}_k \tensor \Anti^\bullet
    \textrm{ and }
    b \in \mathcal{W}_\ell \tensor \Anti^\bullet.
\end{equation}

Following the construction of Fedosov we define the operators $\delta$
and $\delta^*$ by
\begin{equation}
    \label{eq:DeltaDeltaStar}
    \delta
    =
    e^i \wedge \inss(e_i)
    \quad
    \textrm{and}
    \quad
    \delta^*
    =
    e^i \vee \insa(e_i),
\end{equation}
where $\inss$ and $\insa$ are the symmetric and antisymmetric
insertion derivations.  Both maps are graded derivations of $\mu$ with
respect to the antisymmetric degree: $\delta$ lowers the symmetric
degree by one and raises the antisymmetric degree by one, for
$\delta^*$ it is the other way round.  For homogeneous elements
$a \in \Sym^k\lie{g}^* \tensor \Anti^\ell \lie{g}^*$ we define by
\begin{equation}
    \label{eq:deltaInvDef}
    \delta^{-1}(a)
    =
    \begin{cases}
        0,
        & \textrm{if } k + \ell = 0 \\
        \frac{1}{k + \ell}\delta^*(a)
        & \textrm{else,}
  \end{cases}
\end{equation}
and extend this $\ring{R}[[t]]$-linearly.  Notice that this map is not
the inverse of $\delta$, instead we have the following properties:
\begin{lemma}
    \label{lem:Poincare}%
    For $\delta$, $\delta^*$ and $\delta^{-1}$ defined above, we have
    $\delta^2 = (\delta^*)^2 = (\delta^{-1})^2 = 0$ and
    \begin{align*}
        \label{eq:Poincare}
        \delta\delta^{-1} + \delta^{-1}\delta + \sigma
        =
        \id,
    \end{align*}
    where $\sigma$ is the projection on the symmetric and
    antisymmetric degree zero.
\end{lemma}
In fact, this can be seen that the polynomial version of the Poincaré
lemma: $\delta$ corresponds to the exterior derivative and
$\delta^{-1}$ is the standard homotopy.

The next step consists in deforming the product $\mu$ into a
noncommutative one: we define the \emph{star product}
$\circs$ for $a, b \in \mathcal{W} \tensor \Anti^\bullet$ by
\begin{equation}
    \label{eq:WeylMoyalStarProduct}
    a \circs  b
    =
    \mu \circ \E^{\frac{t}{2}\mathcal{P}}(a\tensor b),
    \quad
    \textrm{where}
    \quad
    \mathcal{P} = \pi^{ij} \inss(e_i) \tensor \inss(e_j),
\end{equation}
for $\pi^{ij} = r^{ij} + s^{ij}$, where $r^{ij}$ are the coefficients
of the $r$-matrix and $s^{ij} = s(e^i, e^j) \in \ring{R}$ are the
coefficients of a \emph{symmetric} bivector $s \in \Sym^2 \lie{g}$.
When taking $s = 0$ we denote $\circs$ simply by $\circweyl$.
\begin{proposition}
    \label{prop:StarProductAss}%
    The star product $\circs$ is an associative
    $\ring{R}[[t]]$-bilinear product on
    $\mathcal{W} \tensor \Anti^\bullet$ deforming $\mu$ in zeroth
    order of $t$.  Moreover, the maps $\delta$, $\dega$, and $\Deg$
    are graded derivations of $\circs$ of antisymmetric degree $+1$
    for $\delta$ and $0$ for $\dega$ and $\Deg$, respectively.
\end{proposition}
\begin{proof}
    The associativity follows from the fact that the insertion
    derivations are commuting, see
    \cite[Thm.~8]{gerstenhaber:1968a}. The statement about $\delta$,
    $\dega$ and $\Deg$ are immediate verifications.
\end{proof}

Next, we will need the graded commutator with respect to the
antisymmetric degree, denoted by
\begin{equation}
    \label{eq:GradedCommutator}
    \ad(a)(b)
    =
    [a, b]
    =
    a \circs b - (-1)^{k\ell} b \circs a,
\end{equation}
for any $a \in \mathcal{W} \tensor \Anti^k$ and $b \in
\mathcal{W}\tensor \Anti^\ell$ and extended
$\mathbb{K}[[t]]$-bilinearly as usual. Since $\circs$ deforms the
graded commutative product $\mu$, all graded commutators $[a, b]$ will
vanish in zeroth order of $t$. This allows to define graded
derivations $\frac{1}{t} \ad(a)$ of $\circs$.

\begin{lemma}
    \label{lem:Center}%
    An element $a \in \mathcal{W} \tensor \Anti^\bullet$ is central,
    that is $\ad(a) = 0$, if and only if $\degs(a) = 0$.
\end{lemma}

By definition, a \emph{covariant derivative} is an arbitrary bilinear
map
\begin{equation}
    \label{eq:CovariantDerivative}
    \nabla\colon
    \lie{g} \times \lie{g} \ni (X, Y)
    \; \mapsto \;
    \nabla_X Y \in \lie{g}.
\end{equation}
The idea is that in the geometric interpretation the covariant
derivative is uniquely determined by its values on the left invariant
vector fields: we want an invariant covariant derivative and hence it
should take values again in $\lie{g}$. An arbitrary covariant
derivative is called \emph{torsion-free} if
\begin{equation}
    \label{eq:TorsionFree}
    \nabla_X Y - \nabla_Y X - [X, Y] = 0
\end{equation}
for all $X, Y \in \lie{g}$. Having a covariant derivative, we can
extend it to the tensor algebra over $\lie{g}$ by requiring the maps
\begin{equation}
    \label{eq:nablaXTensors}
    \nabla_X\colon
    \Tensor^\bullet \lie{g} \longrightarrow \Tensor^\bullet \lie{g}
\end{equation}
to be derivations for all $X \in \lie{g}$. We also extend $\nabla_X$
to elements in the dual by
\begin{equation}
    \label{eq:nablaXDual}
    (\nabla_X \alpha)(Y) = - \alpha(\nabla_X Y)
\end{equation}
for all $X, Y \in \lie{g}$ and $\alpha \in \lie{g}^*$. Finally, we can
extend $\nabla_X$ to $\Tensor^\bullet \lie{g}^*$ as a derivation,
too. Acting on symmetric or antisymmetric tensors, $\nabla_X$ will
preserve the symmetry type and yields a derivation of the $\vee$- and
$\wedge$-products, respectively. The fact that we extended $\nabla$ as
a derivation in a way which is compatible with natural pairings will
lead to relations like
\begin{equation}
    \label{eq:CommutatorNablaInss}
    [\nabla_X, \inss(Y)] = \inss(\nabla_X Y)
\end{equation}
for all $X, Y \in \lie{g}$ as one can easily check on generators.

Sometimes it will be advantageous to use the basis of $\lie{g}$ for
computations. With respect to the basis we define
the \emph{Christoffel symbols}
\begin{equation}
    \label{eq:ChristoffelSymbols}
    \Gamma_{ij}^k = e^k(\nabla_{e_i} e_j)
\end{equation}
of a covariant derivative, where $i, j, k = 1, \ldots, n$. Clearly,
$\nabla$ is uniquely determined by its Christoffel symbols. Moreover,
$\nabla$ is torsion-free iff
\begin{equation}
    \label{eq:ChristoffelTorsion}
    \Gamma^k_{ij} - \Gamma^k_{ji} = C_{ij}^k
\end{equation}
with the usual structure constants
$C_{ij}^k = e^k([e_i, e_j]) \in \ring{R}$ of the Lie algebra
$\lie{g}$.

As in symplectic geometry, the Hess trick \cite{hess:1981a} shows the
existence of a \emph{symplectic} torsion-free covariant derivative:
\begin{proposition}[Hess trick]
    \label{prop:HessTrick}%
    Let $(\lie{g}, r)$ be a Lie algebra with non-degenerate $r$-matrix
    $r$ and inverse $\omega$. Then there exists a torsion-free
    covariant derivative $\nabla$ such that for all $X \in \lie{g}$ we
    have
    \begin{equation}
        \label{eq:HessTrick}
        \nabla_X \omega = 0
        \quad
        \textrm{and}
        \quad
        \nabla_X r = 0.
    \end{equation}
\end{proposition}
\begin{proof}
    The idea is to start with the half-commutator connection as in the
    geometric case and make it symplectic by means of the Hess trick.
    The covariant derivative
    \begin{equation*}
        \tilde\nabla
        \colon
        \lie{g} \times \lie{g} \ni (X, Y)
        \; \mapsto \;
        \frac{1}{2}[X, Y] \in \lie{g}
    \end{equation*}
    is clearly torsion-free.  Since $\omega$ is non-degenerate, we can
    determine a map $\nabla_X$ uniquely by
    \begin{equation}
        \label{eq:HessConnection}
        \omega(\nabla_X Y,Z)
        =
        \omega(\tilde\nabla_X Y,Z)
        +
        \frac{1}{3}(\tilde{\nabla}_X\omega)(Y,Z)
        +
        \frac{1}{3}(\tilde{\nabla}_Y\omega)(X,Z).
    \end{equation}
    It is then an immediate computation using the closedness
    $\delta_\CE\omega = 0$ of $\omega$, that this map satisfies all
    requirements.
\end{proof}

The curvature $\tilde{R}$ corresponding to $\nabla$ is defined by
\begin{equation}
    \label{eq:CurvatureTilde}
    \tilde{R}\colon
    \lie{g} \times \lie{g} \times \lie{g} \ni (X, Y, Z)
    \; \mapsto \;
    \tilde{R}(X, Y)Z
    =
    \nabla_X \nabla_Y Z - \nabla_Y\nabla_X Z - \nabla_{[X,Y]}Z
    \in \lie{g}
\end{equation}
For a symplectic covariant derivative, we contract $\tilde{R}$ with the
symplectic form $\omega$ and get
\begin{equation}
    \label{eq:Curvature}
    R\colon
    \lie{g} \times \lie{g} \times \lie{g} \times \lie{g}
    \ni (Z, U, X, Y)
    \; \mapsto \;
    \omega(Z, \tilde{R}(X, Y)U)
    \in \ring{R},
\end{equation}
which is symmetric in the first two components and antisymmetric in
the last ones: this follows at once from $\nabla$ being torsion-free
and symplectic. In other words, $R\in \Sym^2(\lie{g}^*) \tensor
\Anti^2\lie{g}^*$ becomes an element of the formal Weyl algebra
satisfying
\begin{equation}
    \label{eq:DegreesCurvature}
    \degs R = 2 R = \Deg R,
    \quad
    \dega R = 2 R,
    \quad
    \textrm{and}
    \quad
    \deg_t R = 0.
\end{equation}

In the following, we will fix a symplectic torsion-free covariant
derivative, the existence of which is granted by
Proposition~\ref{prop:HessTrick}. Since $\nabla_X$ acts on all types
of tensors already, we can use $\nabla$ to define the following
derivation $D$ on the formal Weyl algebra
\begin{equation}
    \label{eq:ConnectionWeylAlgebra}
    D\colon
    \mathcal{W} \tensor \Anti^\bullet \ni (f \tensor \alpha)
    \; \mapsto \;
    \nabla_{e_i} f\tensor e^i \wedge \alpha
    +
    f \tensor e^i\wedge\nabla_{e_i} \alpha
    \in
    \mathcal{W} \tensor \Anti^{\bullet+1}.
\end{equation}
Notice that we do not use the explicit expression of $\nabla$ given in
\eqref{eq:HessConnection}. In fact, any other symplectic torsion-free
covariant derivative will do the job as well.

For every torsion-free covariant derivative $\nabla$ it is easy to
check that
\begin{equation}
    \label{eq:TorsionFreeCEDifferential}
    e^i \wedge \nabla_{e_i}\alpha = \delta_\CE \alpha
\end{equation}
holds for all $\alpha \in \Anti^\bullet \lie{g}^*$: indeed, both sides
define graded derivations of antisymmetric degree $+1$ and coincide on
generators in $\lie{g}^* \subseteq \Anti^\bullet\lie{g}^*$. Therefore,
we can rewrite $D$ as
\begin{equation}
  \label{eq:ExtendedDerivative}
  D (f \tensor \alpha)
  =
  \nabla_{e_i}f \tensor e^i \wedge \alpha
  +
  f \tensor \delta_\CE \alpha.
\end{equation}
From now on, unless clearly stated, we refer to $[\cdot, \cdot]$ as
the super-commutator with respect to the anti-symmetric degree.
\begin{proposition}
    \label{proposition:OperatorProperties}%
    Let $\nabla$ be a symplectic torsion-free covariant derivative.
    If in addition $s$ is covariantly constant, i.e. if
    $\nabla_X s = 0$ for all $X \in \lie{g}$, the map
    $D\colon \mathcal{W} \tensor \Anti^\bullet \longrightarrow
    \mathcal{W} \tensor \Anti^{\bullet+1}$
    is a graded derivation of antisymmetric degree $+1$ of the star
    product $\circs$, i.e.
    \begin{equation}
        \label{eq:CovariantDerivativeWeyl}
        D (a \circs b)
        =
        D(a) \circs b + (-1)^k a \circs D(b)
    \end{equation}
    for $a \in \mathcal{W} \tensor \Anti^k$ and $b \in
    \mathcal{W}\tensor \Anti^\bullet$. In addition, we have
    \begin{equation}
        \label{eq:BianchiId}
        \delta R = 0,
        \quad
        DR = 0,
        \quad
        [\delta, D]
        =
        \delta D + D \delta
        =
        0,
        \quad
        \textrm{and}
        \quad
        D^2
        =
        \tfrac{1}{2}[D, D]
        =
        \tfrac{1}{t}\ad(R).
    \end{equation}
\end{proposition}
\begin{proof}
    For the operator $\mathcal{P}$ from
    \eqref{eq:WeylMoyalStarProduct} we have
    \begin{align*}
        &(\id \tensor \nabla_{e_k} + \nabla_{e_k}\tensor\id)
        \mathcal P(a\tensor b) \\
        &\quad=
        \pi^{ij}\inss(e_i)a\tensor\nabla_{e_k}\inss(e_j)b
        +
        \pi^{ij}\nabla_{e_k}\inss(e_i)a\tensor\inss(e_j)b
        \\
        &\quad\ot{(a)}{=}
        (\pi^{\ell j}\Gamma_{k\ell}^i
        +
        \pi^{i\ell}\Gamma^j_{k\ell})\inss(e_i)a\tensor\inss(e_j)b
        +
        \mathcal P(\id\tensor \nabla_{e_k}
        +
        \nabla_{e_k}\tensor\id)(a\tensor b)
        \\
        &\quad=
        \mathcal P(\id\tensor \nabla_{e_k}
        +
        \nabla_{e_k}\tensor\id)(a\tensor b)
    \end{align*}
    for $a, b \in \mathcal{W}\tensor \Anti^\bullet$.  Here we used the
    relation $[\nabla_X, \inss(Y)] = \inss(\nabla_X Y)$ as well as the
    definition of the Christoffel symbols in ($a$). In the last step
    we used $\pi^{\ell j}\Gamma_{k\ell}^i + \pi^{i\ell}\Gamma^j_{k\ell} =
    0$ which follows from $\nabla (r+s) = 0$.  Therefore we have
    \begin{align*}
        \nabla_{e_i}
        \circ
        \mu \circ \E^{\frac{t}{2}\mathcal P}
        =
        \mu
        \circ
        (\id\tensor \nabla_{e_i} + \nabla_{e_i}\tensor\id)
        \circ
        \E^{\frac{t}{2}\mathcal P}
        =
        \mu
        \circ
        \E^{\frac{t}{2}\mathcal P}
        \circ
        (\id \tensor \nabla_{e_i} + \nabla_{e_i} \tensor \id).
    \end{align*}
    By $\wedge$-multiplying by the corresponding $e^i$'s it follows
    that $D$ is a graded derivation of antisymmetric degree $+1$.  Let
    $f \tensor \alpha \in \mathcal{W} \tensor \Anti^\bullet$. Just
    using the definition of $\delta$, \eqref{eq:ExtendedDerivative}
    and the fact that $\nabla$ is torsion-free we get
    \begin{align*}
        \delta D(f\tensor\alpha)
        &=
        \delta(\nabla_{e_k}f\tensor e^k\wedge\alpha
        +
        f\tensor\delta_\CE \alpha)
        \\
        &
        =
        -D\delta(f \tensor \alpha)
        +
        \tfrac{1}{2}
        (\Gamma^\ell_{ik} - \Gamma^\ell_{ki} - C^\ell_{ik})
        \inss(e_\ell) f \tensor e^i \wedge e^k \wedge \alpha
        \\
        &=
        - D\delta(f \tensor \alpha).
    \end{align*}
    Using a similar computation in coordinates, we get $D^2 =
    \frac{1}{2}[D,D] = \frac{1}{t} \ad(R)$. Finally, from the Jacobi
    identity of the graded commutator we get $\frac{1}{2t} \ad(\delta
    R) = [\delta, [D, D]] = 0$. Hence $\delta R$ is central. Since
    $\delta R$ has symmetric degree $+1$, this can only happen if
    $\delta R = 0$. With the same argument, $0 = [D, [D, D]]$ yields
    that $DR$ is central, which again gives $DR = 0$ by counting
    degrees.
\end{proof}
\begin{remark}
    \label{remark:WeylOrOtherOrdering}%
    In principle, we will mainly be interested in the case $s = 0$ in
    the following. However, if the Lie algebra allows for a
    covariantly constant $s$ it might be interesting to incorporate
    this into the universal construction: already in the abelian case
    this leads to the freedom of choosing a different ordering than
    the Weyl ordering (total symmetrization). Here in particular the
    Wick ordering is of significance due to the better positivity
    properties, see \cite{bursztyn.waldmann:2000a} for a universal
    deformation formula in this context.
\end{remark}

The core of Fedosov's construction is now to turn $-\delta + D$ into a
differential: due to the curvature $R$ the derivation $- \delta + D$
is not a differential directly. Nevertheless, from the above
discussion we know that it is an inner derivation. Hence the idea is
to compensate the defect of being a differential by inner derivations,
leading to the following statement:
\begin{proposition}
    \label{proposition:FedosovDerivation}%
    Let $\Omega \in t\Anti^2 \lie g^*[[t]]$ be a series of
    $\delta_\CE$-closed two-forms. Then there is a unique
    $\varrho \in \mathcal{W}_2\tensor\Anti^1$, such that
    \begin{align}
        \label{eq:x-property2}
        \delta \varrho
        =
        R + D\varrho + \tfrac{1}{t} \varrho \circs \varrho
        +
        \Omega
    \end{align}
    and
    \begin{align}
        \label{eq:x-property1}
        \delta^{-1} \varrho
        =
        0.
    \end{align}
    Moreover, the derivation
    $\FedosovD = -\delta + D + \frac{1}{t}\ad(\varrho)$ satisfies
    $\FedosovD^2 = 0$.
\end{proposition}
\begin{proof}
    Let us first assume that \eqref{eq:x-property2} is satisfied and
    apply $\delta^{-1}$ to \eqref{eq:x-property1}.  This yields
    \begin{align*}
        \delta^{-1}\delta \varrho
        =
        \delta^{-1}\left(
            R + Dx + \tfrac{1}{t} \varrho \circs \varrho + \Omega
        \right).
    \end{align*}
    From the Poincaré Lemma as in \eqref{eq:Poincare} we have
    \begin{align}
        \label{eq:x-necessaryprop}
        \varrho
        =
        \delta^{-1}
        \left(
            R
            + D\varrho + \tfrac{1}{t} \varrho \circs \varrho
            + \Omega
        \right).
    \end{align}
    Let us define the operator $B \colon \mathcal{W} \tensor \Anti^1
    \longrightarrow \mathcal{W} \tensor \Anti^1$ by
    \begin{align*}
        B(a)
        =
        \delta^{-1}(R + Da + \tfrac{1}{t}a\circs a + \Omega).
    \end{align*}
    Thus the solutions of \eqref{eq:x-property1} coincide with the
    fixed points of the operator $B$.  Now we want to show that $B$
    has indeed a unique fixed point.  By a careful but straightforward
    counting of degrees we see that $B$ maps $\mathcal{W}_2 \tensor
    \Anti^1$ into $\mathcal{W}_2 \tensor \Anti^1$.
    Second, we note that $B$ is a contraction with respect to the
    total degree. Indeed, for $a, a' \in \mathcal{W}_2 \tensor\Anti^1$
    with $a - a' \in \mathcal{W}_k \tensor \Anti^1$ we have
    \begin{align*}
        B(a)- B(a')
        &=
        \delta^{-1} D(a - a')
        +
        \tfrac{1}{t}
        \left(a \circs a - a' \circs a'\right) \\
        &=
        \delta^{-1} D(a - a')
        +
        \tfrac{1}{t}\delta^{-1}\left(
            (a - a') \circs a' + a \circs (a - a')
        \right).
    \end{align*}
    The first term $\delta^{-1}D(a - a')$ is an element of
    $\mathcal{W}_{k+1}\tensor\Anti^1$, because $D$ does not change the
    total degree and $\delta^{-1}$ increases it by $+1$.  Since $\Deg$
    is a $\circs$-derivation and since $a, a'$ have total degree at
    least $2$ and their difference has total degree at least $k$, the
    second term has total degree at least $k+1$, as $\frac{1}{t}$ has
    total degree $-2$ but $\delta^{-1}$ raises the total degree by
    $+1$. This allows to apply the Banach fixed-point theorem for the
    complete filtration by the total degree: we have a unique
    fixed-point $B(\varrho) = \varrho$ with
    $\varrho \in \mathcal{W}_2 \tensor \Anti^1$, i.e. $\varrho$
    satisfies \eqref{eq:x-necessaryprop}.  Finally, we show that this
    $\varrho$ fulfills \eqref{eq:x-property1}. Define
    \begin{align*}
        A
        =
        \delta \varrho
        - R
        - D\varrho
        - \tfrac{1}{t} \varrho \circs \varrho
        - \Omega.
    \end{align*}
    Apply $\delta$ to $A$ and using
    Prop.~\ref{proposition:OperatorProperties} we obtain
    \begin{align*}
        \delta A
        &=
        - \delta D \varrho
        - \tfrac{1}{t}
        \left(
            \delta \varrho \circs \varrho
            -
            \varrho \circs \delta \varrho
        \right) \\
        &=
        D\delta \varrho + \tfrac{1}{t} \ad(\varrho)\delta \varrho
        \\
        &=
        D\left(
            A
            + R
            + D\varrho
            + \tfrac{1}{t} \varrho \circs \varrho
            + \Omega
        \right)
        +
        \tfrac{1}{t}\ad(\varrho)
        \left(
            A
            + R
            + D\varrho
            + \tfrac{1}{t} \varrho \circs \varrho
            + \Omega
        \right)
        \\
        &\ot{(a)}{=}
        DA + \tfrac{1}{t}\ad(\varrho)(A).
    \end{align*}
    In ($a$) we used the fact that
    $(-\delta + D + \frac{1}{t}\ad (\varrho)) (R + D \varrho +
    \frac{1}{t} \varrho \circs \varrho + \Omega) = 0$,
    which can be seen as a version of the second Bianchi identity for
    $- \delta + D + \frac{1}{t}\ad(\varrho)$. This follows by an
    explicit computation for arbitrary $\varrho$. On the other hand
    \begin{align*}
        \delta^{-1}A
        =
        \delta^{-1} \left(
            \delta \varrho
            - R
            - D\varrho
            - \tfrac{1}{t} \varrho \circs \varrho
            -\Omega
        \right)
        =
        \delta^{-1}\delta \varrho - \varrho
        =
        \delta\delta^{-1} \varrho
        =
        0
    \end{align*}
    for $\varrho$ being the fixed-point of the operator $B$. In other
    words,
    \begin{align*}
        A
        =
        \delta^{-1}\delta A
        =
        \delta^{-1}\left(DA + \tfrac{1}{t}\ad(\varrho)(A)\right)
    \end{align*}
    is a fixed-point of the operator
    $K \colon \mathcal{W} \tensor \Anti^\bullet \longrightarrow
    \mathcal{W} \tensor \Anti^\bullet$ defined by
    \begin{align*}
        K a
        =
        \delta^{-1}\left(Da + \tfrac{1}{t}\ad(\varrho)(a)\right).
    \end{align*}
    Using an analogous argument as above, this operator is a
    contraction with respect to the total degree, and has a unique
    fixed-point. Finally, since $K$ is linear the fixed point has to
    be zero, which means that $A = 0$.
\end{proof}
\begin{remark}
    \label{remark:RecursiveConstruction}%
    It is important to note that the above construction of the element
    $\varrho$, which will be the crucial ingredient in the universal
    deformation formula below, is a fairly explicit recursion formula.
    Writing $\varrho = \sum_{r=3}^\infty \varrho^{(r)}$ with
    components $\varrho^{(r)}$ of homogeneous total degree
    $\Deg \varrho^{(r)} = r \varrho^{(r)}$ we see that
    $\varrho^{(3)} = \delta^{-1} (R + t \Omega_1)$ and
    \begin{equation}
        \label{eq:VarrhoIteratively}
        \varrho^{(r+3)}
        =
        \delta^{-1}\left(
            D \varrho^{(r+2)}
            +
            \tfrac{1}{t}
            \sum_{\ell = 1}^{r-1}
            \varrho^{(\ell+2)} \circs \varrho^{(r+2-\ell)}
            +
            \Omega^{(r+2)}
        \right),
    \end{equation}
    where $\Omega^{(2k)} = t^k \Omega_k$ for $k \in \mathbb{N}$ and
    $\Omega^{(2k+1)} = 0$.  Moreover, if we find a \emph{flat}
    $\nabla$, i.e. if $R = 0$, then for trivial $\Omega = 0$ we have
    $\varrho = 0$ as solution.
\end{remark}

%
%

\section{Universal Deformation Formula}
\label{sec:UniversalDeformationFormula}

Let us consider a triangular Lie algebra $(\lie g, r)$ acting on a
generic associative algebra $(\algebra A, \mu_{\algebra A})$ via
derivations. We denote by $\acts$ the corresponding Hopf algebra
action $\uea{\lie{g}} \longrightarrow \End(\algebra A)$.  In the
following we refer to
\begin{align*}
  \algebra{A} \tensor \mathcal{W} \tensor \Anti^\bullet
  =
  \prod_{k=0}^\infty
  \left(
      \algebra{A}
      \tensor \Sym^k \lie{g}^*
      \tensor \Anti^\bullet \lie{g}^*
  \right)[[t]]
\end{align*}
as the \emph{enlarged Fedosov algebra}.  The operators defined in the
previous section are extended to
$\algebra{A} \tensor \mathcal{W} \tensor \Anti^\bullet$ by acting
trivially on the $\algebra{A}$-factor and as before on the
$\mathcal{W} \tensor \Anti^\bullet$-factor.

The deformed product $\circs$ on $\mathcal{W} \tensor \Anti^\bullet$
together with the product $\mu_{\algebra{A}}$ of $\algebra{A}$ yields
a new (deformed) $\ring{R}[[t]]$-bilinear product $\msA$ for the
extended Fedosov algebra. Explicitly, on factorizing tensors we have
\begin{equation}
    \label{eq:ExtendedProduct}
    \msA\left(
        \xi_1 \tensor f_1 \tensor \alpha_1,
        \xi_2 \tensor f_2 \tensor\alpha_2
    \right)
    =
    (\xi_1 \cdot \xi_2)
    \tensor
    (f_1\tensor\alpha_1) \circs (f_2\tensor\alpha_2),
\end{equation}
where $\xi_1, \xi_2 \in \algebra{A}$,
$f_1, f_2 \in \Sym^\bullet\lie{g}^*$ and
$\alpha_1, \alpha_2 \in \Anti^\bullet\lie{g}^*$. We simply write
$\xi_1 \cdot \xi_2$ for the (undeformed) product $\mu_{\algebra{A}}$
of $\algebra{A}$. Clearly, this new product $\msA$ is again associative.

As new ingredient we use the action $\acts$ to define the operator
$L_{\algebra{A}}\colon \algebra{A} \tensor \mathcal{W} \tensor
\Anti^\bullet \longrightarrow \algebra{A} \tensor \mathcal{W} \tensor
\Anti^\bullet$ by
\begin{align}
    \label{eq:LeftMult}
    L_{\algebra{A}}(\xi \tensor f \tensor \alpha)
    =
    e_i \acts \xi \tensor f \tensor e^i\wedge \alpha
\end{align}
on factorizing elements and extend it $\ring{R}[[t]]$-linearly as
usual.  Since the action of Lie algebra elements is by derivations, we
see that $L_{\algebra{A}}$ is a derivation of
$\algebra{A} \tensor \mathcal{W} \tensor \Anti^\bullet$ of
antisymmetric degree $+1$. The sum
\begin{equation}
    \label{eq:ExtendedFedosovDerivation}
    \FedosovA
    =
    L_{\algebra A} + \FedosovD
\end{equation}
is thus still a derivation of antisymmetric degree $+1$ which we call
the \emph{extended Fedosov derivation}. It turns out to be a
differential, too:
\begin{lemma}
    The map
    $\FedosovA = L_{\algebra A} + \FedosovD$
    squares to zero.
\end{lemma}
\begin{proof}
    First, we observe that
    $\FedosovA^2 = L_{\algebra A}^2 + [\FedosovD,
    L_{\algebra A}]$,
    because $\FedosovD^2 = 0$. Next, since $\acts$ is a Lie
    algebra action, we immediately obtain
    \[
    L_{\algebra A}^2(\xi\tensor f\tensor \alpha)
    =
    \frac{1}{2}C_{ij}^k e_k \acts \xi
    \tensor f
    \tensor e^i\wedge e^j \wedge \alpha
    \]
    on factorizing elements. We clearly have
    $[\delta, L_{\algebra{A}}] = 0 = [\ad(\varrho), L_{\algebra{A}}]$
    since the maps act on different tensor factors. It remains to
    compute the only nontrivial term in
    $[\FedosovD, L_{\algebra A}] = [D, L_{\algebra{A}}]$.  Using
    $\delta_\CE e^k = - \frac{1}{2} C^k_{ij} e^i \wedge e^j$, this
    results immediately in
    $[D, L_{\algebra{A}}] = - L_{\algebra{A}}^2$.
\end{proof}

The cohomology of this differential turns out to be almost trivial: we
only have a nontrivial contribution in antisymmetric degree $0$, the
kernel of $\FedosovA$. In higher antisymmetric
degrees, the following homotopy formula shows that the cohomology is
trivial:
\begin{proposition}
    \label{proposition:CohomologyExtendedFedosov}%
    The operator
    \begin{equation}
        \label{eq:DAHomotopy}
        \FedosovA^{-1}
        =
        \delta^{-1}
        \frac{1}
        {
          \id
          -
          \left[
              \delta^{-1},
              D + L_{\algebra{A}} + \frac{1}{t}\ad(\varrho)
          \right]
        }
    \end{equation}
    is a well-defined $\ring{R}[[t]]$-linear endomorphism of
    $\algebra A\tensor \mathcal{W}\tensor\Anti^\bullet$ and we have
    \begin{equation}
        a
        =
        \FedosovA \FedosovA^{-1} a
        +
        \FedosovA^{-1}\FedosovA a
        +
        \frac{1}
        {
          \id
          -
          \left[
              \delta^{-1},
              D + L_{\algebra A} + \frac{1}{t}\ad(\varrho)
          \right]
        }
        \sigma(a).
    \end{equation}
    for all $a \in \algebra{A} \tensor \mathcal{W} \tensor \Anti^\bullet$.
\end{proposition}
\begin{proof}
    Let us denote by $A$ the operator
    $[\delta^{-1}, D + L_{\algebra A} +
    \frac{1}{t}\ad(\varrho)]$.
    Since it increases the total degree by $+1$, the geometric series
    $(\id - A)^{-1}$ is well-defined as a formal series in the total
    degree. We start with the Poincaré lemma \ref{eq:Poincare} and get
    \begin{equation}
        \label{eq:GeometricSeries}
        - \FedosovA \delta^{-1}a
        - \delta^{-1} \FedosovA a
        + \sigma(a)
        =
        (\id - A) a,
    \end{equation}
    since $\mathcal{D}_{\mathcal{A}}$ deforms the differential
    $-\delta$ by higher order terms in the total degree. The usual
    homological perturbation argument then gives \eqref{eq:DAHomotopy}
    by a standard computation, see
    e.g. \cite[Prop.~6.4.17]{waldmann:2007a} for this computation.
\end{proof}
\begin{corollary}
    \label{cor:TrivialCohomology}%
    Let $a \in \algebra A \tensor \mathcal{W} \tensor \Anti^0$.  Then
    $\FedosovA a = 0$ if and only if
    \begin{equation}
        \label{eq:aInKernelDA}
        a
        =
        \frac{1}
        {
          \id
          -
          \left[
              \delta^{-1},
              D + L_{\algebra A} + \frac{1}{t}\ad(\varrho)
          \right]
        }
        \sigma(a).
    \end{equation}
\end{corollary}

Since the element
$a \in \algebra{A} \tensor \mathcal{W} \tensor \Anti^0$ is completely
determined in the symmetric and antisymmetric degree $0$, we can use
it to define the extended Fedosov Taylor series.
\begin{definition}[Extended Fedosov Taylor series]
    \label{def:ExtFedosovTaylorSeries}%
    Given the extended Fedosov derivation
    $\FedosovA = -\delta + D + L_{\algebra{A}} +
    \frac{1}{t}\ad(\varrho)$,
    the extended Fedosov Taylor series of $\xi \in \algebra{A}[[t]]$
    is defined by
    \begin{equation}
        \label{eq:ExtendedFedosovTaylorSeries}
        \tau_{\algebra{A}}(\xi)
        =
        \frac{1}
        {
          \id
          -
          \left[
              \delta^{-1},
              D + L_{\algebra{A}} + \frac{1}{t}\ad(\varrho)
          \right]
        }
        \xi.
    \end{equation}
\end{definition}
\begin{lemma}
    \label{lem:TauAsASeries}%
    For $\xi \in \algebra A[[t]]$ we have
    \begin{equation}
        \label{eq:SigmaTau}
        \sigma(\tau_{\algebra{A}}(\xi))
        =
        \xi.
    \end{equation}
    Moreover, the map
    $\tau_{\algebra{A}}\colon \algebra{A}[[t]] \longrightarrow
    \ker\FedosovA \cap \ker \dega$
    is a $\ring{R}[[t]]$-linear isomorphism starting with
    \begin{equation}
        \label{eq:LowestDegOrdersTau}
        \tau_{\algebra{A}}(\xi)
        =
        \sum_{k = 0}^\infty
        \left[
            \delta^{-1},
            D + L_{\algebra{A}} + \tfrac{1}{t} \ad(\varrho)
        \right]^k (\xi)
        =
        \xi \tensor 1 \tensor 1
        +
        e_i \acts\xi \tensor e^i \tensor 1
        + \cdots
    \end{equation}
    in zeroth and first order of the total degree.
\end{lemma}
\begin{proof}
    The isomorphism property follows directly from
    Corollary~\ref{cor:TrivialCohomology}.  The commutator
    $[\delta^{-1}, D + L_{\algebra{A}} + \frac{1}{t}\ad(\varrho)]$
    raises the total degree at least by one, thus the zeroth and first
    order terms in the total degree come from the terms with $k = 0$
    and $k = 1$ in the geometric series in
    \eqref{eq:LowestDegOrdersTau}.  Here it is easy to see that the
    only non-trivial contribution is
    \[
    \left[
        \delta^{-1},
        D + L_{\algebra{A}} + \tfrac{1}{t} \ad(\varrho)
    \right] \xi
    =
    L_{\algebra{A}} \xi,
    \]
    proving the claim in \eqref{eq:LowestDegOrdersTau}. Note that
    already for $k = 2$ we get also contributions of $S$ and
    $\ad(\varrho)$.
\end{proof}

Given the $\ring{R}[[t]]$-linear isomorphism
$\tau_{\algebra{A}}\colon \algebra{A}[[t]] \longrightarrow \ker
\FedosovA \cap \ker \dega$
we can turn $\algebra{A}[[t]]$ into an algebra by pulling back the
deformed product: note that the kernel of a derivation is always a
subalgebra and hence the intersection $\ker\FedosovA \cap \ker \dega$
is also a subalgebra.  This allows us to obtain a universal
deformation formula for any $\uea{\lie{g}}$-module algebra
$\algebra{A}$:
\begin{theorem}[Universal deformation formula]
    \label{theorem:StarProduct}%
    Let $\lie{g}$ be a Lie algebra with non-degenerate
    $r$-matrix. Moreover, let $s \in \Sym^2 \lie{g}$ be such that
    there exists a symplectic torsion-free covariant derivative
    $\nabla$ with $s$ being covariantly constant.  Consider then
    $\pi = r + s$. Finally, let $\Omega \in t \Anti^2 \lie{g}^*[[t]]$
    be a formal series of $\delta_\CE$-closed two-forms.  Then for
    every associative algebra $\algebra{A}$ with action of $\lie{g}$
    by derivations one obtains an associative deformation
    $m_\star^{\algebra{A}}\colon \algebra{A}[[t]]\times\algebra A[[t]]
    \longrightarrow \algebra{A}[[t]]$ by
    \begin{equation}
        \label{eq:TheRealUDF}
        m_\star^{\algebra{A}} (\xi, \eta)
        =
        \sigma\left(
            \msA
            (\tau_{\algebra{A}}(\xi), \tau_{\algebra{A}}(\eta))
        \right)
    \end{equation}
    Writing simply $\star = \star_{\Omega, \nabla, s}$ for this new
    product, one has
    \begin{equation}
        \label{eq:TheStarProduct}
        \xi \star \eta
        =
        \xi \cdot \eta
        +
        \frac{t}{2} \pi^{ij} (e_i \acts\xi) \cdot (e_j \acts \eta)
        +
        \mathcal{O}(t^2)
    \end{equation}
    for $\xi, \eta \in \algebra{A}$.
\end{theorem}
\begin{proof}
    The product $m_\star^{\algebra{A}}$ is associative, because $\msA$
    is associative and $\tau_{\algebra{A}}$ is an isomorphism onto a
    subalgebra with inverse $\sigma$.  The second part is a direct
    consequence of Lemma~\ref{lem:TauAsASeries}.
\end{proof}
\begin{remark}
    The above theorem can be further generalized by observing that
    given a Poisson structure on $\algebra{A}$ induced by a generic
    bivector on $\lie{g}$, we can reduce to the quotient
    $\lie{g} / \ker\acts$ and obtain an $r$-matrix on the quotient,
    inducing the same Poisson structure.
\end{remark}

%
%

\section{Universal Construction for Drinfel'd Twists}
\label{sec:UniversalDeformationFormulaTwist}

Let us consider the particular case in which $\algebra{A}$ is the
tensor algebra $(\Tensor^\bullet(\uea{\lie{g}}), \tensor)$.  In this
case, we denote by $L$ the operator $L_{\Tensor^\bullet(\uea{\lie
    g})}\colon \Tensor^\bullet(\uea{\lie{g}})\tensor
\mathcal{W}\tensor\Anti^\bullet \longrightarrow
\Tensor^\bullet(\uea{\lie{g}})\tensor \mathcal{W}\tensor\Anti^\bullet
$, which is given by
\begin{equation}
    \label{eq:LforUg}
    L_{\Tensor^\bullet(\uea{\lie{g}})}(\xi\tensor f\tensor \alpha)
    =
    L_{e_i}\xi\tensor f\tensor e^i\wedge \alpha.
\end{equation}
Here $L_{e_i}$ is the left multiplication in $\uea{\lie{g}}$ of the
element $e_i$ extended as a derivation of the tensor product.  Note
that it is independent of the choice of the basis in $\lie{g}$.

Applying the results discussed in the last section, we obtain a star
product for the tensor algebra over $\uea{\lie{g}}$ as a particular
case of Theorem~\ref{theorem:StarProduct}:
\begin{corollary}
    \label{cor:StarProduct}%
    The map $m_\star\colon \Tensor^\bullet(\uea{\lie
      g})[[t]]\times\Tensor^\bullet(\uea{\lie{g}})[[t]]
    \longrightarrow \Tensor^\bullet(\uea{\lie{g}})[[t]]$ defined by
    \begin{equation}
        \label{eq:StarProductForTensorUg}
        m_\star (\xi,\eta)
        =
        \xi\star \eta
        =
        \sigma (\ms(\tau(\xi),\tau(\eta)))
    \end{equation}
    is an associative product and
    \begin{equation}
        \label{eq:ClassLimTensorAlgebraStar}
        \xi\star \eta
        =
        \xi\tensor \eta
        +
        \frac{t}{2} \pi^{ij}L_{e_i}\xi\tensor L_{e_j}\eta
        +
        \mathcal{O}(t^2)
    \end{equation}
    for $\xi, \eta \in \Tensor^\bullet(\uea{\lie{g}})$.
\end{corollary}

In the following we prove that the star product $m_\star$ defined
above allows to construct a formal Drinfel'd twist. Let us define, for
any linear map
\begin{equation}
    \label{eq:PhiUktoUl}
    \Phi \colon \uea{\lie{g}}^{\tensor k}
    \longrightarrow
    \uea{\lie{g}}^{\tensor \ell},
\end{equation}
the \emph{lifted} map
\begin{equation}
    \label{eq:LiftedMap}
    \Phi^\lift \colon
    \uea{\lie{g}}^{\tensor k}
    \tensor
    \mathcal{W}\tensor\Anti^\bullet
    \ni \xi\tensor f\tensor \alpha
    \; \mapsto \;
    \Phi(\xi)\tensor f\tensor\alpha \in
    \uea{\lie{g}}^{\tensor \ell}
    \tensor
    \mathcal{W}\tensor\Anti^\bullet,
\end{equation}
obeying the following simple properties:
\begin{lemma}
    \label{lem:LiftingAndStructures}%
    Let $\Phi\colon \uea{\lie{g}}^{\tensor k} \longrightarrow
    \uea{\lie{g}}^{\tensor \ell}$ and $\Psi\colon
    \uea{\lie{g}}^{\tensor m} \longrightarrow \uea{\lie{g}}^{\tensor
      n}$ be linear maps.
    \begin{lemmalist}
    \item The lifted map $\Phi^\lift$ commutes with $\delta$,
        $\delta^{-1}$, $D$, and $\ad(x)$ for all
        $x \in \mathcal{W}\tensor \Anti^\bullet$.
    \item We have
        \begin{equation}
            \label{eq:ListCommutesWithSigma}
            \Phi \circ
            \sigma\at{
              \uea{\lie{g}}^{\tensor k}
              \tensor
              \mathcal{W}\tensor \Anti^\bullet
            }
            =
            \sigma\at{
              \uea{\lie{g}}^{\tensor \ell}
              \tensor
              \mathcal{W}\tensor\Anti^\bullet
            }
            \circ \Phi^\lift.
        \end{equation}
    \item We have
        \begin{equation}
            \label{eq:LiftAndProduct}
            (\Phi\tensor\Psi)^\lift \ms(a_1,a_2)
            =
            \ms(\Phi^\lift(a_1), \Psi^\lift(a_2)),
        \end{equation}
        for any
        $a_1 \in \uea{\lie{g}}^{\tensor k} \tensor \mathcal{W} \tensor
        \Anti^\bullet$
        and
        $a_2 \in \uea{\lie{g}}^{\tensor m} \tensor \mathcal{W} \tensor
        \Anti^\bullet$.
    \end{lemmalist}
\end{lemma}

Let $\eta \in \uea{\lie{g}}^{\tensor k}[[t]]$ be given. Then we can
consider the right multiplication by $\eta$ using the algebra
structure of $\uea{\lie{g}}^{\tensor k}[[t]]$ coming from the
universal enveloping algebra as a map
\begin{equation}
    \label{eq:RightMultiplication}
    \cdot \eta\colon
    \uea{\lie{g}}^{\tensor k}\ni \xi
    \; \mapsto \;
    \xi \cdot \eta \in  \uea{\lie{g}}^{\tensor k}.
\end{equation}
To this map we can apply the above lifting process and extend it this
way to a $\ring{R}[[t]]$-linear map such that on factorizing
elements
\begin{equation}
    \label{eq:LiftedRightMultiplication}
    \cdot \eta\colon
    \uea{\lie{g}}^{\tensor k}
    \tensor
    \mathcal{W} \tensor \Anti^\bullet
    \ni \xi \tensor f \tensor \alpha
    \; \mapsto \;
    (\xi \cdot \eta) \tensor f \tensor \alpha
    \in  \uea{\lie{g}}^{\tensor k},
\end{equation}
where we simply write $\cdot \eta$ instead of $(\cdot \eta)^\lift$.
Note that $a \cdot \eta$ is only defined if the tensor degrees $k$ of
$\eta \in \Tensor^k(\uea{\lie{g}})$ and $a$ coincide since we use the
algebra structure inherited from the universal enveloping algebra.

In the following we denote by $\FedosovTU$ the derivation
$\mathscr{D}_{\Tensor^\bullet(\uea{\lie{g}})}$ as obtained in
\eqref{eq:ExtendedFedosovDerivation}. We collect some properties how
the lifted right multiplications match with the extended Fedosov
derivation:
\begin{lemma}
    \label{lem:RactsProperties}%
    \begin{lemmalist}
    \item\label{item:RactsI} For any
        $a \in \Tensor^k(\uea{\lie{g}}) \tensor \mathcal{W} \tensor
        \Anti^\bullet$
        and $\xi \in \Tensor^k(\uea{\lie{g}})[[t]]$, we have
        $\FedosovTU(a \cdot \xi) = \FedosovTU(a) \cdot \xi$
    \item\label{item:RactsII} The extended Fedosov Taylor series
        $\tau$ preserves the tensor degree of elements in
        $\Tensor^\bullet(\uea{\lie{g}})$.
    \item\label{item:RactsIII} For any
        $\xi, \eta \in \Tensor^k(\uea{\lie{g}})[[t]]$, we have
        $\tau(\xi\cdot\eta) = \tau(\xi) \cdot \eta$.
    \item\label{item:RactsVI} For any
        $a_1 \in \Tensor^k(\uea{\lie{g}}) \tensor \mathcal{W} \tensor
        \Anti^\bullet$
        and
        $a_2 \in \Tensor^\ell(\uea{\lie{g}}) \tensor \mathcal{W}
        \tensor \Anti^\bullet$
        as well as $\eta_1 \in \Tensor^k(\uea{\lie{g}})[[t]]$ and
        $\eta_2 \in \Tensor^\ell(\uea{\lie{g}})[[t]]$, we have
        $\ms(a_1 \cdot \eta_1, a_2 \cdot_l \eta_2) = \ms(a_1, a_2)
        \cdot (\eta_1 \tensor \eta_2)$.
  \end{lemmalist}
\end{lemma}
\begin{proof}
    Let
    $\xi \tensor a \in \Tensor^k(\uea{\lie{g}}) \tensor \mathcal{W}
    \tensor \Anti^\bullet$
    and $\eta \in \Tensor^k(\uea{\lie{g}})$ then we have
    \begin{align*}
        \FedosovTU((\xi \tensor a) \cdot \eta)
        &=
        \FedosovTU((\xi \cdot \eta) \tensor a)
        \\
        &=
        L_{e_i} (\xi\cdot\eta) \tensor e^i\wedge a
        +
        (\xi \cdot \eta) \tensor \FedosovD(a)
        \\
        &=
        (L_{e_i}(\xi) \tensor e^i \wedge a) \cdot \eta
        +
        (\xi \tensor \FedosovD(a)) \cdot \eta
        \\
        &=
        \FedosovTU(a) \cdot \eta.
    \end{align*}
    This proves the first claim.  The second claim follows immediately
    from the fact that all operators defining $\tau$ do not change the
    tensor degree.  In order to prove the claim
    \refitem{item:RactsIII}, let us consider
    $\xi, \eta \in \Tensor^k(\uea{\lie{g}})[[t]]$. Then we have
    \begin{equation*}
        \FedosovTU(\tau(\xi)\cdot \eta)
        =
        \FedosovTU(\tau(\xi))\cdot \eta
        =
        0,
    \end{equation*}
    according to \refitem{item:RactsI}. Thus,
    $\tau(\xi) \cdot \eta \in \ker{\FedosovTU} \cap \ker \dega$ and
    therefore
    \begin{equation*}
        \tau(\xi) \cdot \eta
        =
        \tau(\sigma(\tau(\xi)\cdot \eta))
        =
        \tau(\sigma(\tau(\xi))\cdot \eta)
        =
        \tau(\xi\cdot\eta).
    \end{equation*}
    Finally, to prove the last claim we choose
    $\xi_1 \tensor f_1 \in \Tensor^k(\uea{\lie{g}}) \tensor
    \mathcal{W} \tensor \Anti^\bullet$
    and
    $\xi_2 \tensor f_2 \in \Tensor^\ell(\uea{\lie{g}}) \tensor
    \mathcal{W} \tensor \Anti^\bullet$
    as well as $\eta_1 \in \Tensor^k(\uea{\lie{g}}) [[t]]$ and
    $\eta_2 \in \Tensor^\ell(\uea{\lie{g}})[[t]]$. We obtain
    \begin{align*}
        \ms\left(
            (\xi_1\tensor f_1) \cdot \eta_1,
            (\xi_2\tensor f_2) \cdot \eta_2
        \right)
        &=
        \ms\left(
            (\xi_1 \cdot \eta_1) \tensor f_1,
            (\xi_2 \cdot \eta_2) \tensor f_2
        \right) \\
        &=
        \left(
            (\xi_1 \cdot \eta_1) \tensor (\xi_2 \cdot \eta_2)
        \right)
        \tensor (f_1 \circs f_2) \\
        &=
        \left(
            (\xi_1 \tensor \xi_2) \cdot (\eta_1 \tensor \eta_2)
        \right)
        \tensor (f_1 \circs f_2)\\
        &=
        \left(
            (\xi_1 \tensor \xi_2) \tensor (f_1 \circs f_2)
        \right) \cdot (\eta_1 \tensor \eta_2).
    \end{align*}
    This concludes the proof.
\end{proof}

From the above lemma, we observe that the isomorphism $\tau$ can be
computed for any element $\xi \in \Tensor^k(\uea{\lie{g}})[[t]]$ via
\begin{equation}
    \label{eq:tauxiExplicitly}
    \tau(\xi)
    =
    \tau(1^{\tensor k} \cdot \xi)
    =
    \tau(1^{\tensor k }) \cdot \xi,
\end{equation}
where $1 \in \uea{\lie{g}}$ is the unit element of the universal
enveloping algebra.  Moreover, from
Lemma~\ref{lem:LiftingAndStructures}, we have
\begin{equation}
    \label{eq:StarPr}
    \xi \star \eta
    =
    \sigma(\ms(\tau(\xi)\tensor \tau(\eta)))
    =
    (1^{\tensor k} \star 1^{\tensor \ell}) \cdot (\xi\tensor \eta)
\end{equation}
for $\xi \in \Tensor^k(\uea{\lie{g}})[[t]]$ and
$\eta \in \Tensor^\ell(\uea{\lie{g}})[[t]]$.  Thus $\star$ is entirely
determined by the values on tensor powers of the unit element of the
universal enveloping algebra. Note that the unit of $\star$ is the
unit element in $\ring{R} \subseteq \Tensor^\bullet(\uea{\lie{g}})$
of the \emph{tensor algebra} but not $1 \in \uea{\lie{g}}$.
\begin{lemma}
    \label{lem:LiftedHopfStrucures}%
    Let
    $\Delta\colon \uea{\lie{g}}[[t]] \longrightarrow
    \uea{\lie{g}}^{\tensor 2}[[t]]$
    be the coproduct of $\uea{\lie{g}}[[t]]$ and
    $\epsilon\colon \uea{\lie{g}}\to \ring{R}[[t]]$ the counit.
    \begin{lemmalist}
    \item We have
        \begin{equation}
            \label{eq:LcommutesWithDelta}
            L\at{
              \uea{\lie{g}}^{\tensor 2}
              \tensor \mathcal{W} \tensor \Anti^\bullet
            }
            \circ
            \Delta^\lift
            =
            \Delta^\lift
            \circ
            L\at{\uea{\lie{g}}\tensor \mathcal{W}\tensor \Anti^\bullet}.
        \end{equation}
    \item For the Fedosov-Taylor series one has
        \begin{equation}
            \Delta^\lift \circ \tau = \tau \circ \Delta.
        \end{equation}
    \item We have
        \begin{equation}
            \label{eq:EpsilonOfL}
            \epsilon^\lift
            \circ
            L\at{\uea{\lie{g}}\tensor \mathcal{W}\tensor \Anti^\bullet}
            =
            0.
        \end{equation}
    \item For the Fedosov-Taylor series one has
        \begin{equation}
            \label{eq:EpsilonListTau}
            \epsilon^\lift \circ \tau =\epsilon.
        \end{equation}
    \end{lemmalist}
\end{lemma}
\begin{proof}
    Let
    $\xi \tensor f \tensor \alpha \in \uea{\lie{g}} \tensor
    \mathcal{W} \tensor \Anti^\bullet$ then we get
    \begin{align*}
        \Delta^\lift L(\xi \tensor f \tensor \alpha)
        &=
        \Delta^\lift\left(
            L_{e_i}(\xi) \tensor f \tensor e^i \wedge \alpha
        \right) \\
        &=
        \Delta^\lift\left(
            e_i \xi \tensor f \tensor e^i \wedge \alpha
        \right)\\
        &=
        \Delta(e_i \xi) \tensor f \tensor e^i \wedge \alpha \\
        &=
        \Delta (e_i) \cdot \Delta(\xi)
        \tensor f \tensor e^i \wedge \alpha \\
        &=
        (e_i\tensor 1 + 1\tensor e_i) \cdot \Delta(\xi)
        \tensor f \tensor e^i \wedge \alpha \\
        &=
        L_{e_i}(\Delta(\xi)) \tensor f \tensor e^i \wedge \alpha \\
        &=
        L \Delta^\lift (\xi\tensor f\tensor \alpha),
    \end{align*}
    since we extended the left multiplication by $e_i$ as a
    \emph{derivation} of the tensor product to higher tensor powers.
    Hence all the operators appearing in $\tau$ \emph{commute} with
    $\Delta^\lift$ and therefore we get the the second part.
    Similarly, we get
    \begin{align*}
        \epsilon^\lift (L(\xi \tensor f \tensor \alpha)
        &=
        \epsilon^\lift (e_i \xi \tensor f \tensor e^i \wedge \alpha)
        \\
        &=
        \epsilon(e_i \xi) \tensor f \tensor e^i \wedge \alpha \\
        &=
        \epsilon(e_i) \epsilon(\xi)
        \tensor f \tensor e^i \wedge \alpha \\
        &=
        0,
    \end{align*}
    where we used that $\epsilon$ vanishes on primitive elements of
    $\uea{\lie{g}}$. Since $\epsilon^\lift$ commutes with all other
    operators $\delta^{-1}$, $D$ and $\ad(\varrho)$ according to
    Lemma~\ref{lem:LiftingAndStructures}, we first get
    \[
    \epsilon^\lift
    \circ
    \left[\delta^{-1}, D + L + \tfrac{1}{t}\ad(\varrho)\right]
    =
    \left[\delta^{-1}, D + \tfrac{1}{t}\ad(\varrho)\right]
    \circ
    \epsilon^\lift.
    \]
    Hence for $\xi \in \uea{\lie{g}}[[t]]$ we have
    \begin{align*}
        \epsilon^\lift \tau(\xi)
        &=
        \epsilon^\lift
        \left(
            \sum_{k=0}^\infty
            \left[
                \delta^{-1}, D + L + \tfrac{1}{t} \ad(\varrho)
            \right]^k
            \xi
        \right) \\
        &=
        \sum_{k=0}^\infty
        \left[
            \delta^{-1}, D + \tfrac{1}{t} \ad(\varrho)
        \right]^k
        \epsilon^\lift (\xi) \\
        &=
        \epsilon(\xi),
    \end{align*}
    since $\epsilon^\lift(\xi) = \epsilon(\xi)$ is just a constant and
    hence unaffected by all the operators in the series. Thus only the
    zeroth term remains.
\end{proof}

This is now the last ingredient to show that the element $1 \star 1$
is the twist we are looking for:
\begin{theorem}
    \label{Thm:twist}
    The element $1 \star 1 \in \uea{\lie{g}}^{\tensor 2}[[t]]$ is a
    twist such that
    \begin{equation}
        \label{eq:TheTwistAtLast}
        1\star 1
        =
        1\tensor 1 + \frac{t}{2} \pi + \mathcal{O}(t^2).
    \end{equation}
\end{theorem}
\begin{proof}
    First we see that
    \begin{align*}
        (\Delta \tensor \id)(1\star 1)
        &=
        (\Delta\tensor\id) \sigma (\ms(\tau(1),\tau(1))) \\
        &=
        \sigma\left(
            (\Delta \tensor \id)^\lift
            (\ms(\tau(1), \tau(1)))
        \right) \\
        &=
        \sigma\left(
            \ms(\Delta^\lift \tau(1), \tau(1))
        \right) \\
        &=
        \sigma (\ms(\tau(\Delta (1)), \tau(1))) \\
        &=
        \sigma(\ms(\tau(1\tensor 1), \tau(1))) \\
        &=
        (1 \tensor 1) \star 1.
    \end{align*}
    Similarly, we get
    $(\id \tensor \Delta)(1\star 1) = 1 \star (1\tensor 1)$.  Thus,
    using the associativity of $\star$ we obtain the first condition
    \eqref{eq:TwistConditionI} for a twist as follows,
    \begin{align*}
        (\Delta \tensor \id) (1\star 1) \cdot ((1\star 1) \tensor 1)
        &=
        ((1 \tensor 1)\star 1) \cdot ((1\star 1) \tensor 1) \\
        &=
        (1 \star 1) \star 1 \\
        &=
        1 \star (1 \star 1) \\
        &=
        (\id \tensor \Delta)(1 \star 1)
        \cdot (1 \tensor (1 \star 1)).
    \end{align*}
    To check the normalization condition \eqref{eq:TwistConditionII}
    we use Lemma~\ref{lem:LiftingAndStructures} and
    Lemma~\ref{lem:LiftedHopfStrucures} again to get
    \begin{align*}
        (\epsilon \tensor \id)(1\star 1)
        &=
        (\epsilon \tensor \id) \sigma(\ms(\tau(1), \tau(1))) \\
        &=
        \sigma\left(
            (\epsilon \tensor \id)^\lift
            (\ms(\tau(1), \tau(1)))
        \right) \\
        &=
        \sigma\left(
            (\ms(\epsilon^\lift\tau(1), \tau(1)))
        \right) \\
        &=
        \sigma\left(
            (\ms(\epsilon(1), \tau(1)))
        \right) \\
        &=
        \epsilon(1) \sigma(\tau(1)) \\
        &=
        1,
    \end{align*}
    since $\epsilon(1)$ is the unit element of $\ring{R}$ and thus
    the unit element of $\Tensor^\bullet(\uea{\lie{g}})$, which serves
    as unit element for $\ms$ as well.  Similarly we obtain
    $(\id\tensor\epsilon)(1\star 1) = 1$.  Finally, the facts that the
    first term in $t$ of $1 \star 1$ is given by $\pi$ and that zero
    term in $t$ is $1 \tensor 1$ follow from
    Corollary~\ref{cor:StarProduct}.
\end{proof}
\begin{remark}
    \label{remark:WhatWeGotNow}%
    From now on we refer to $1\star 1$ as the \emph{Fedosov twist}
    \begin{equation}
        \label{eq:FedosovTwistDefinition}
        \twist{F}_{\Omega, \nabla, s} = 1 \star 1,
    \end{equation}
    corresponding to the choice of the $\delta_\CE$-closed form
    $\Omega$, the choice of the torsion-free symplectic covariant
    derivative and the choice of the covariantly constant $s$. In the
    following we will be mainly interested in the dependence of
    $\twist{F}_{\Omega, \nabla, s}$ on the two-forms $\Omega$ and
    hence we shall write $\twist{F}_\Omega$ for simplicity.  We also
    note that for $s = 0$ and $\Omega = 0$ we have a \emph{preferred}
    choice for $\nabla$, namely the one obtained from the Hess trick
    out of the half-commutator covariant derivative as described in
    Proposition~\ref{prop:HessTrick}. This gives a \emph{canonical
      twist} $\twist{F}_0$ quantizing $r$.
\end{remark}

The results discussed above allow us to give an alternative proof of
the Drinfel'd theorem \cite{drinfeld:1983a}, stating the existence of
twists for every $r$-matrix:
\begin{corollary}[Drinfel'd]
    \label{corollary:ExtendedDrinfeld}%
    Let $(\lie{g},r) $ be a Lie algebra with $r$-matrix over a field
    $\mathbb{K}$ with characteristic $0$. Then there exists a formal
    twist $\twist {F}\in(\uea{\lie{g}}\tensor\uea{\lie{g}})[[t]]$,
    such that
    \begin{align*}
        \twist{F}
        =
        1 \tensor 1 + \frac{t}{2} r + \mathcal{O}(t^2).
    \end{align*}
\end{corollary}

To conclude this section we consider the question whether the two
approaches of universal deformation formulas actually coincide: on the
one hand we know that every twist gives a universal deformation
formula by \eqref{eq:TheUDF}. On the other hand, we have constructed
directly a universal deformation formula \eqref{eq:TheRealUDF} in
Theorem~\ref{theorem:StarProduct} based on the Fedosov
construction. Since we also get a twist from the Fedosov construction,
we are interested in the consistence of the two constructions. In
order to answer this question, we need some preparation. Hence let
$\algebra{A}$ be an algebra with action of $\lie{g}$ by derivations as
before. Then we define the map
\begin{equation}
    \bullet\colon
    \uea{\lie{g}} \tensor \mathcal{W} \tensor \Anti^\bullet
    \times
    \algebra{A}
    \ni
    (\xi \tensor \alpha, a)
    \; \mapsto \;
    (\xi \tensor \alpha) \bullet a
    =
    \xi \acts a \tensor \alpha
    \in
    \algebra{A} \tensor \mathcal{W} \tensor \Anti^\bullet
\end{equation}
for any $a \in \algebra{A} $ and
$\alpha \in \mathcal{W} \tensor \Anti^\bullet$. Then the following
algebraic properties are obtained by a straightforward computation:
\begin{lemma}
    \label{lemma:ModuleStructures}%
    For any $\xi \in \uea{\lie{g}}$,
    $\alpha \in \mathcal{W} \tensor \Anti^\bullet$ and
    $a \in \algebra{A}$ we have
    \begin{lemmalist}
    \item
        $\sigma ((\xi \tensor \alpha) \bullet a) = \sigma (\xi \tensor
        \alpha) \acts a$,
    \item
        $L_{\algebra{A}} (\xi \acts a \tensor \alpha) = L (\xi \tensor
        \alpha) \bullet a$,
    \item $\tau_{\algebra{A}}(a) = \tau (1) \bullet a$,
    \item
        $\msA (\xi_1 \tensor a_1 \tensor \alpha_1 , \xi_2 \tensor a_2
        \tensor \alpha_2) = (\mu_{\algebra{A}} \tensor \id \tensor
        \id) (\ms (\xi_1 \tensor \alpha_1 , \xi_2 \tensor
        \alpha_2)\bullet (a_1 \tensor a_2))$.
    \end{lemmalist}
\end{lemma}
For matching parameters $\Omega$, $\nabla$, and $s$ of the Fedosov
construction, the two approaches coincide:
\begin{proposition}
    \label{prop:Coincide}%
    For fixed choices of $\Omega$, $\nabla$, and $s$ and for any
    $a, b \in \algebra{A}$ we have
    \begin{equation}
        a \star_{\Omega, \nabla, s} b
        =
        a \star_{\twist{F}_{\Omega, \nabla, s}} b.
    \end{equation}
\end{proposition}
\begin{proof}
    This is now just a matter of computation. We have
    \begin{align*}
        a \star b
        &=
        \sigma\left(
            \msA (\tau_{\algebra{A}}(a) \tensor \tau_{\algebra{A}}(b))
        \right)
        \\
        &\ot{(a)}{=}
        \sigma\left(
            \ms ((\tau(1) \tensor \tau(1)) \bullet (a \tensor b))
        \right)
        \\
        &\ot{(b)}{=}
        \mu_{\algebra{A}}\left(
            \sigma (\ms (\tau (1)\tensor \tau (1)))
            \acts
            (a \tensor b)
        \right)
        \\
        &=
        \mu_{\algebra{A}} ((1 \star 1) \acts  (a \tensor b))
        \\
        &=
        a \star_{\twist{F}} b,
    \end{align*}
    where in $(a)$ we use the third claim of the above lemma and in
    $(b)$ the first and the fourth.
\end{proof}

%
%

\section{Classification of Drinfel'd Twists}
\label{sec:Classification}

In this section we discuss the classification of twists on universal
enveloping algebras for a given Lie algebra $\lie{g}$, with
non-degenerate $r$-matrix. Recall that two twists $\twist{F}$ and
$\twist{F}'$ are said to be \emph{equivalent} and denoted by
$\twist{F}\sim \twist{F}'$ if there exists an element
$S\in \uea{\lie{g}}[[t]]$, with $S = 1 + \mathcal{O}(t)$ and
$\epsilon(S) = 1$ such that
\begin{equation}
    \label{eq:EquivalentTwist}
    \Delta(S) \twist{F}'
    =
    \twist{F} (S \tensor S).
\end{equation}
In the following we prove that the set of equivalence classes of
twists $\mathrm{Twist}(\uea{\lie{g}},r)$ with fixed $r$-matrix $r$ is
in bijection to the formal series in the second Chevalley-Eilenberg
cohomology $\mathrm{H}_\CE^2(\lie{g})[[t]]$.

We will fix the choice of $\nabla$ and the symmetric part $s$ in the
Fedosov construction. Then the cohomological equivalence of the
two-forms in the construction yields equivalent twists. In fact, an
equivalence can even be computed recursively:
\begin{lemma}
    \label{lem:FedosovEquiv}%
    Let $\varrho$ and $\varrho'$ be the two elements in
    $\mathcal{W}_2 \tensor \Anti^1$ uniquely determined from
    Proposition~\ref{proposition:FedosovDerivation}, corresponding to
    two closed two-forms $\Omega, \Omega' \in t\Anti^2\lie{g}^*[[t]]$,
    respectively, and let $\Omega - \Omega' = \delta_\CE C$ for a
    fixed $C \in t\lie{g}^*[[t]]$.  Then there is a unique solution
    $h\in\mathcal{W}_3\tensor\Anti^0$ of
    \begin{equation}
        \label{eq:FedosovEquivConstr}
        h
        =
        C \tensor 1
        +
        \delta^{-1}\left(
            D h
            -
            \frac{1}{t}\ad(\varrho) h
            -
            \frac{\frac{1}{t}\ad(h)}
            {\exp(\frac{1}{t}\ad(h)) - \id}
            (\varrho' - \varrho)
        \right)
        \quad
        \textrm{and}
        \quad
        \sigma(h) = 0.
    \end{equation}
    For this $h$ we have
    \begin{align*}
        \FedosovD'
        =
        \mathcal{A}_h \FedosovD \mathcal{A}_{-h},
    \end{align*}
    with $\mathcal{A}_h = \exp(\frac{1}{t}\ad(h))$ being an
    automorphism of $\circs$.
\end{lemma}
\begin{proof}
    In the context of the Fedosov construction it is well-known that
    cohomologous two-forms yield equivalent star products. The above
    approach with the explicit formula for $h$ follows the arguments
    of \cite[Lemma~3.5]{reichert.waldmann:2016a} which is based on
    \cite[Sect.~3.5.1.1]{neumaier:2001a}.
\end{proof}
\begin{lemma}
    \label{lem:CohomologousEquiv}%
    Let $\Omega, \Omega' \in t\Anti^2\lie{g}^*[[t]]$ be
    $\delta_\CE$-cohomologous.  Then the corresponding Fedosov twists
    are equivalent.
\end{lemma}
\begin{proof}
    By assumption, we can find an element $C \in t\lie g^*[[t]]$, such
    that $\Omega - \Omega' = \delta_\CE C$.  From
    Lemma~\ref{lem:FedosovEquiv} we get an element
    $h \in \mathcal{W}_3 \tensor \Anti^0$ such that
    $\FedosovTU'_F = \mathcal{A}_h \FedosovD \mathcal{A}_{-h}$.
    An easy computation shows that $\mathcal{A}_h$ commutes with $L$,
    therefore we have
    \begin{equation*}
        \FedosovTU'
        =
        \mathcal{A}_h\FedosovTU\mathcal{A}_{-h}.
    \end{equation*}
    Thus, $\mathcal{A}_h$ is an automorphism of $\ms$ with
    $\mathcal{A}_h \colon \ker \FedosovTU \longrightarrow
    \ker\FedosovTU'$
    being a bijection between the two kernels.  Let us consider the
    map
    \begin{equation*}
        S_h\colon
        \Tensor^\bullet (\uea{\lie{g}})[[t]]
        \ni \xi
        \; \mapsto \;
        (\sigma\circ\mathcal{A}_h\circ\tau)(\xi) \in
        \Tensor^\bullet (\uea{\lie{g}})[[t]],
    \end{equation*}
    which is defines an equivalence of star products, i.e.
    \begin{align}
        \label{eq:StarAutomorphism}
        S_h (\xi \star \eta)
        =
        S_h(\xi) \star' S_h(\eta)
    \end{align}
    for any $\xi, \eta \in \Tensor^\bullet(\uea{\lie{g}})[[t]]$. Let
    $\xi, \eta \in \uea{\lie{g}}$, then using
    Lemma~\ref{lem:RactsProperties} we have
    \begin{align*}
        S_h(\xi\tensor \eta)
        &=
        (\sigma\circ\mathcal{A}_h\circ\tau)(\xi\tensor \eta)\\
        &=
        (\sigma\circ\mathcal{A}_h)
        (\tau(1\tensor 1) \cdot (\xi\tensor \eta)) \\
        &=
        \sigma\left(
            (\mathcal{A}_h(\tau(1 \tensor 1)))
            \cdot (\xi \tensor \eta)
        \right) \\
        &=
        \sigma(\mathcal{A}_h(\tau(1 \tensor 1)))
        \cdot (\xi \tensor \eta) \\
        &=
        \sigma\left(\mathcal{A}_h(\Delta^\lift \tau(1))\right)
        \cdot(\xi\tensor\eta) \\
        &=
        \Delta\left(
            \sigma(\mathcal{A}_h(\tau(1)))
        \right)
        \cdot(\xi\tensor\eta) \\
        &=
        \Delta(S_h(1)) \cdot (\xi \tensor \eta).
    \end{align*}
    From the linearity of $S_h$ we immediately get
    $S_h(\xi \star \eta) = \Delta(S_h(1))(\xi \star \eta)$.  Now,
    putting $\xi = \eta = 1$ in \eqref{eq:StarAutomorphism} and using
    \eqref{eq:StarPr} we obtain
    \begin{equation*}
        \Delta(S_h(1)) \cdot (1\star 1)
        =
        S_h(1\star 1)
        =
        S_h(1) \star' S_h(1)
        =
        (1 \star' 1) \cdot (S_h(1)\tensor S_h(1)).
    \end{equation*}
    Thus, the twists $\twist{F}_{\Omega} = 1 \star
    1$ and $\twist{F}_{\Omega'} = 1 \star' 1$ are equivalent since we have
    \begin{equation*}
      \epsilon (S_h (1))
      =
      1.
    \end{equation*}

\end{proof}
\begin{lemma}
    \label{lem:FedosovTwistProperties}%
    Let $\Omega \in t\Anti^2\lie{g}^*$ with $\delta_\CE\Omega=0$,
    $x$ the element in $\mathcal{W}_2 \tensor
    \Anti^1$ uniquely determined from
    Proposition~\ref{proposition:FedosovDerivation} and
    $\twist{F}_\Omega$ the corresponding Fedosov twist.
    \begin{lemmalist}
    \item The lowest total degree of $\varrho$, where $\Omega_k$
        appears, is $2k+1$, and we have
        \begin{equation}
            \label{eq:FirstOmegaInVarrho}
            \varrho^{(2k+1)}
            =
            t^k \delta^{-1}\Omega_k
            +
            \textrm{terms not containing }
            \Omega_k.
        \end{equation}
    \item For $\xi \in \Tensor^\bullet (\uea{\lie{g}})$ the lowest
        total degree of $\tau(\xi)$, where $\Omega_k$ appears, is
        $2k+1$, and we have
        \begin{equation}
            \label{eq:FirstOmegaInTau}
            \tau(\xi)^{(2k+1)}
            =
            \frac{t^k}{2}
            \left(
                e_i \tensor \insa((e^i)^\sharp) \Omega_k
            \right)
            +
            \textrm{terms not containing }
            \Omega_k.
        \end{equation}
    \item The lowest $t$-degree of $\twist{F}_\Omega$, where
        $\Omega_k$ appears, is $k+1$, and we have
        \begin{align*}
            (F_\Omega)_{k+1}
            =
            -\frac{1}{2}(\Omega_k)^\sharp
            +
            \textrm{terms not containing }
            \Omega_k.
        \end{align*}
    \item The map $\Omega \mapsto \twist{F}_\Omega$ is injective.
    \end{lemmalist}
\end{lemma}
\begin{proof}
    The proof uses the recursion formula for $\varrho$ as well as the
    explicit formulas for $\tau$ and $\star$ and consists in a careful
    counting of degrees. It follows the same lines of
    \cite[Thm.~6.4.29]{waldmann:2007a}.
\end{proof}
\begin{lemma}
    \label{lemma:CohomologousOmegas}%
    Let $\twist{F}_\Omega$ and $\twist{F}_{\Omega'}$ be two equivalent
    Fedosov twists corresponding to the closed two-forms
    $\Omega, \Omega' \in t\Anti^2\lie{g}^*$. Then there exists an
    element $C\in t\lie{g}^*[[t]]$, such that
    $\delta_\CE C = \Omega - \Omega'$.
\end{lemma}
\begin{proof}
    We can assume that $\Omega$ and $\Omega'$ coincide up to order
    $k-1$ for $k \in \mathbb{N}$, since they coincide at order
    $0$. Due to Lemma~\ref{lem:FedosovTwistProperties}, we have
    \begin{equation*}
        (F_\Omega)_{i}
        =
        (F_{\Omega'})_{i}
    \end{equation*}
    for any $i \in \{0, \ldots, k\}$ and
    \begin{equation*}
        (F_{\Omega})_{k+1} - (F_{\Omega'})_{k+1}
        =
        \frac{1}{2}(-\Omega_k^\sharp +{\Omega'}_k^\sharp).
    \end{equation*}
    From Lemma~\ref{lem:SkewsymmetricTwistExact}, we know that we can
    find an element $\xi \in \lie{g}^*$, such that
    \begin{align*}
        ([(F_{\Omega})_{k+1} - (F_{\Omega'})_{k+1}])^\flat
        =
        - \Omega_k^\sharp + {\Omega'}_k^\sharp
        =
        \delta_\CE\xi,
    \end{align*}
    where by $[(F_{\Omega})_{k+1} - (F_{\Omega'})_{k+1}]$ we denote
    the skew-symmetrization of
    $(F_{\Omega})_{k+1} - (F_{\Omega'})_{k+1}$.  Let us define
    $\hat\Omega = \Omega - t^k\delta_\CE\xi$. From
    Lemma~\ref{lem:FedosovTwistProperties} we see that
    \begin{equation*}
        (F_{\hat\Omega})_{k+1}-(F_{\Omega'})_{k+1}
        =
        0.
    \end{equation*}
    Therefore the two twists $\twist{F}_{\hat\Omega}$ and
    $\twist{F}_{\Omega'}$ coincide up to order $k+1$. Finally, since
    $\twist{F}_{\hat\Omega}$ and $\twist{F}_{\Omega}$ are equivalent
    (from Lemma~\ref{lem:CohomologousEquiv}) and $\twist{F}_{\Omega}$
    and $\twist{F}_{\Omega'}$ are equivalent by assumption, the two
    twists $\twist{F}_{\hat\Omega}$ and $\twist{F}_{\Omega'}$ are also
    equivalent.  By induction, we find an element
    $C \in t\lie{g}^*[[t]]$, such that
    \begin{align*}
        \twist{F}_{\Omega + \delta_\CE C}
        =
        \twist{F}_{\Omega'},
    \end{align*}
    and therefore, from Lemma~\ref{lem:FedosovTwistProperties},
    $\Omega + \delta_\CE C = \Omega'$.
\end{proof}
\begin{lemma}
    \label{lemma:AllTwistAreFedosov}%
    Let $\twist{F}\in (\uea{\lie{g}}\tensor \uea{\lie{g}})[[t]]$ be a
    formal twist with $r$-matrix $r$. Then there exists a Fedosov
    twist $\twist{F}_{\Omega}$, such that
    $\twist{F}\sim \twist{F}_\Omega$.
\end{lemma}
\begin{proof}
    Let $\twist{F}\in (\uea{\lie{g}} \tensor \uea{\lie{g}})[[t]]$ be a
    given twist. We can assume that there is a Fedosov twist
    $\twist{F}_\Omega$, which is equivalent to $\twist{F}$ up to order
    $k$. Therefore we find a $\hat{\twist {F}}$ such that
    $\hat{\twist {F}}$ is equivalent to $\twist{F}$ and coincides with
    $\twist{F}_\Omega$ up to order $k$.  Due to
    Lemma~\ref{lem:SkewsymmetricTwistExact}, we can find an element
    $\xi\in \lie g^*$, such that
    \begin{equation*}
        [(F_\Omega)_{k+1} - \hat{F}_{k+1})]
        =
        (\delta_\CE \xi)^\sharp.
    \end{equation*}
    From Lemma \ref{lem:CohomologousEquiv}, the twist
    $\twist{F}_{\Omega'}$ corresponding to
    $\Omega' = \Omega - t^k\delta_\CE \xi$ is equivalent to
    $\twist{F}_{\Omega}$.  Moreover, $\twist{F}_{\Omega'}$ coincides
    with $\hat{\twist {F}}$ up to order $k$, since
    $\twist{F}_{\Omega'}$ coincides with $\twist{F}_{\Omega}$ and
    \begin{equation*}
        (F_{\Omega'})_{k+1}
        =
        (F_\Omega)_{k+1} + \frac{1}{2}\delta_\CE \xi.
    \end{equation*}
    Therefore the skew-symmetric part of
    $(F_{\Omega'})_{k+1}-\hat{F}_{k+1}$ is vanishing and this
    difference is exact with respect to the differential defined in
    \eqref{eq:HKRDiff}.  Applying Lemma~\ref{lem:EquivalenceOrder}, we
    can see that $\twist{F}_{\Omega'}$ is equivalent to
    $\hat{\twist {F}}$ up to order $k+1$. The claim follows by
    induction.
\end{proof}

Summing up all the above lemmas we obtain the following
characterization of the equivalence classes of twists:
\begin{theorem}[Classification of twists]
    \label{theorem:Classification}%
    Let $\lie{g}$ be a Lie algebra over $\ring{R}$ such that $\lie{g}$
    is free and finite-dimensional and let $r \in \Anti^2\lie{g}$ be a
    classical $r$-matrix such that $\sharp$ is bijective.  Then the
    set of equivalence classes of twists
    $\mathrm{Twist}(\uea{\lie{g}}, r)$ with $r$-matrix $r$ is in
    bijection to $\mathrm{H}_\CE^2(\lie{g})[[t]]$ via
    $\Omega \mapsto \twist{F}_\Omega$.
\end{theorem}
It is important to remark that even for an abelian Lie algebra
$\lie{g}$ the second Chevalley-Eilenberg cohomology
$\mathrm{H}_\CE^2(\lie{g})[[t]]$ is different from zero. Thus, not all
twists are equivalent.  An example of a Lie algebra with trivial
$\mathrm{H}_\CE^2(\lie{g})[[t]]$ is the two-dimensional non-abelian
Lie algebra:
\begin{example}[$ax+b$]
    Let us consider the two-dimensional Lie algebra given by the
    $\ring{R}$-span of the elements $X, Y \in \lie{g}$ fulfilling
    \begin{equation}
        \label{eq:abPlusbLieAlgebra}
        [X, Y] = Y,
    \end{equation}
    with $r$-matrix $r = X \wedge Y$. We denote the dual basis of
    $\lie{g}^*$ by $\{X^*, Y^*\}$. Since $\lie{g}$ is two-dimensional,
    all elements of $\Anti^2\lie{g}^*$ are a multiple of
    $X^* \wedge Y^*$, which is closed for dimensional reasons.  For
    $Y^*$ we have
    \begin{equation}
        \label{eq:XstarYstarExact}
        (\delta_\CE Y^*)(X, Y)
        =
        - Y^*([X, Y])
        =
        -Y^*(Y)
        =
        -1.
    \end{equation}
    Therefore $\delta_\CE Y^* = -X^* \wedge Y^*$ and we obtain
    $\mathrm{H}_\CE^2(\lie{g}) = \{0\}$.  From
    Theorem~\ref{theorem:Classification} we can therefore conclude that
    all twists with $r$-matrix $r$ of $\lie{g}$ are equivalent.
\end{example}
\begin{remark}[Original construction of Drinfel'd]
    \label{remark:DrinfeldConstruction}%
    Let us briefly recall the original construction of Drinfel'd from
    \cite[Thm.~6]{drinfeld:1983a}: as a first step he uses the inverse
    $B \in \Anti^2 \lie{g}^*$ of $r$ as a $2$-cocycle to extend
    $\lie{g}$ to $\tilde{\lie{g}} = \lie{g} \oplus \mathbb{R}$ by
    considering the new bracket
    \begin{equation}
        \label{eq:NewBracket}
        [(X, \lambda), (X', \lambda')]_{\tilde{\lie{g}}}
        =
        ([X, X']_{\lie{g}}, B(X, X'))
    \end{equation}
    where $X, X' \in \lie{g}$ and $\lambda, \lambda' \in
    \mathbb{R}$. On $\tilde{\lie{g}}^*$ one has the canonical star
    product quantizing the linear Poisson structure $\star_{DG}$
    according to Drinfel'd and Gutt \cite{gutt:1983a}. Inside
    $\tilde{\lie{g}}^*$ one has an affine subspace defined by $H =
    \lie{g}^* + \ell_0$ where $\ell_0$ is the linear functional
    $\ell_0\colon \tilde{\lie{g}} \ni (X, \lambda) \mapsto
    \lambda$. Since the extension is central, $\star_{DG}$ turns out
    to be tangential to $H$, therefore it restricts to an associative
    star product on $H$. In a final step, Drinfel'd then uses a local
    diffeomorphism $G \longrightarrow H$ by mapping $g$ to
    $\operatorname{Ad}_{g^{-1}}^* \ell_0$ to pull-back the star
    product to $G$, which turns out to be left-invariant. By
    \cite[Thm.~1]{drinfeld:1983a} this gives a twist. Without major
    modification it should be possible to include also closed higher
    order terms $\Omega \in t \Anti^2 \lie{g}^*[[t]]$ by considering
    $B + \Omega$ instead. We conjecture that
    \begin{enumerate}
    \item this gives all possible classes of Drinfel'd twists by
        modifying his construction including $\Omega$,
    \item the resulting classification matches the classification by
        our Fedosov construction.
    \end{enumerate}
    Note that a direct comparison of the two approaches will be
    nontrivial due to the presence of the combinatorics in the BCH
    formula inside $\star_{DG}$ in the Drinfel'd construction on the
    one hand and the recursion in our Fedosov approach on the other
    hand. We will come back to this in a future project.
\end{remark}

%
%

\section{Hermitian and Completely Positive Deformations}
\label{sec:HermitianCPDeformations}

In this section we include now aspects of positivity into the picture:
in addition, let $\ring{R}$ be now an ordered ring and set
$\ring{C} = \ring{R}(\I)$ where $\I^2 = -1$. In $\ring{C}$ we have a
complex conjugation as usual, denoted by $z \mapsto \cc{z}$. The Lie
algebra $\lie{g}$ will now be a Lie algebra over $\ring{R}$, still
begin free as a $\ring{R}$-module with finite dimension.

The formal power series $\ring{R}[[ t ]]$ are then again an ordered
ring in the usual way and we have
$\ring{C}[[ t ]] = (\ring{R}[[ t ]])(\I)$. Moreover, we consider a
$^*$-algebra $\algebra{A}$ over $\ring{C}$ which we would like to
deform. Here we are interested in \emph{Hermitian} deformations
$\star$, where we require
\begin{equation}
    \label{eq:HermitianDeformation}
    (a \star b)^* = b^* \star a^*
\end{equation}
for all $a, b \in \algebra{A}[[ t ]]$.

Instead of the universal enveloping algebra directly, we consider now
the complexified universal enveloping algebra
$\ueac{\lie{g}} = \uea{\lie{g}} \tensor[\ring{R}] \ring{C} =
\uea{\lie{g}_{\ring{C}}}$
where $\lie{g}_{\ring{C}} = \lie{g} \tensor[\ring{R}] \ring{C}$ is the
complexified Lie algebra. Then this is a $^*$-Hopf algebra where the
$^*$-involution is determined by the requirement
\begin{equation}
    \label{eq:HopfStarAlgebra}
    X^* = - X
\end{equation}
for $X \in \lie{g}$, i.e. the elements of $\lie{g}$ are
\emph{anti-Hermitian}.  The needed compatibility of the action of
$\lie{g}$ on $\algebra{A}$ with the $^*$-involution is then
\begin{equation}
    \label{eq:StarAction}
    (\xi \acts a)^* = S(\xi)^* \acts a^*
\end{equation}
for all $\xi \in \ueac{\lie{g}}$ and $a \in \algebra{A}$. This is
equivalent to $(X \acts a)^* = X \acts a^*$ for $X \in \lie{g}$.  We
also set the elements of $\lie{g}^* \subseteq \lie{g}_{\ring{C}}^*$ to
be \emph{anti-Hermitian}.

In a first step we extend the complex conjugation to tensor powers of
$\lie{g}_{\ring{C}}^*$ and hence to the complexified Fedosov algebra
\begin{equation}
    \mathcal{W}_{\ring{C}} \tensor \Anti^\bullet_{\ring{C}}
    =
    \left(
        \prod_{k=0}^\infty
        \Sym^k\lie{g}_{\ring{C}}^*
        \tensor
        \Anti^\bullet \lie{g}_{\ring{C}}^*
    \right)[[ t ]]
\end{equation}
and obtain a (graded) $^*$-involution, i.e.
\begin{equation}
    \label{eq:FedosovInvolution}
    ((f \tensor \alpha) \cdot (g \tensor \beta))^*
    =
    (-1)^{ab} (g \tensor \beta)^* \cdot (f \tensor \alpha)^*,
\end{equation}
where $a$ and $b$ are the antisymmetric degrees of $\alpha$ and
$\beta$, respectively.

Let $\pi \in \lie{g}_{\ring{C}} \tensor \lie{g}_{\ring{C}}$ have
antisymmetric part $\pi_- \in \Anti^2 \lie{g}_{\ring{C}}$ and
symmetric part $\pi_+ \in \Anti^2 \lie{g}_{\ring{C}}$. Then we have
for the corresponding operator $\mathcal{P}_\pi$ as in
\eqref{eq:WeylMoyalStarProduct}
\begin{equation}
    \mathsf{T} \circ \cc{\mathcal{P}_\pi (a \tensor b)}
    =
    \mathcal{P}_{\tilde{\pi}} \circ \mathsf{T} (\cc{a} \tensor \cc{b})
\end{equation}
where $\tilde{\pi} = \cc{\pi}_+ - \cc{\pi}_-$. In particular, we have
$\tilde{\pi} = \pi$ iff $\pi_+$ is Hermitian and $\pi_-$ is
anti-Hermitian. We set $t = \I t $ for the formal parameter as in
the previous sections, i.e. we want to treat $t$ as imaginary. Then we
arrive at the following statement:
\begin{lemma}
    \label{lemma:HermitianFiberwise}%
    Let $\pi = \pi_+ + \pi_- \in \lie{g}_{\ring{C}} \tensor
    \lie{g}_{\ring{C}}$. Then the fiberwise product
    \begin{equation}
        \label{eq:HermitianFiberwise}
        a \circs  b
        =
        \mu \circ \E^{\frac{\I t }{2}\mathcal{P}_\pi}(a\tensor b)
    \end{equation}
    satisfies $(a \circs b)^* = (-1)^{ab} b^* \circ a^*$ iff $\pi_+$
    is anti-Hermitian and $\pi_-$ is Hermitian.
\end{lemma}
This lemma is now the motivation to take a \emph{real} classical
$r$-matrix
$r \in \Anti^2 \lie{g} \subseteq \Anti^2 \lie{g}_{\ring{C}}$.
Moreover, writing the symmetric part of $\pi$ as $\pi_+ = \I s$ then
$s = \cc{s} \in \Sym^2 \lie{g}$ is Hermitian as well. In the following
we shall assume that these reality condition are satisfied.

It is now not very surprising that with such a Poisson tensor $\pi$ on
$\lie{g}$ we can achieve a Hermitian deformation of a $^*$-algebra
$\algebra{A}$ by the Fedosov construction. We summarize the relevant
properties in the following proposition:
\begin{proposition}
    \label{proposition:HermitianFedosov}%
    Let $\pi = r + \I s$ with a real strongly non-degenerate
    $r$-matrix $r \in \Anti^2 \lie{g}$ and a real symmetric
    $s \in \Sym^2 \lie{g}$ such that there exists a symplectic
    torsion-free covariant derivative $\nabla$ for $\lie{g}$ with
    $\nabla s = 0$.
    \begin{propositionlist}
    \item The operators $\delta$, $\delta^{-1}$, and $\sigma$ are
        real.
    \item The operator $D$ is real and $D^2 = \frac{1}{\I t } \ad(R)$
        with a Hermitian curvature $R = R^*$.
    \item Suppose that
        $\Omega = \Omega^* \in \Anti^2 \lie{g}^*_{\ring{C}}[[ t ]]$
        is a formal series of Hermitian $\delta_\CE$-closed
        two-forms. Then the unique
        $\varrho \in \mathcal{W}_2\tensor\Anti^1$ with
        \begin{equation}
            \label{eq:varrhoHermitian}
            \delta \varrho
            =
            R + D\varrho
            +
            \tfrac{1}{\I t } \varrho \circs \varrho + \Omega
        \end{equation}
        and $\delta^{-1} \varrho = 0$ is Hermitian, too. In this case,
        the Fedosov derivative
        $\FedosovD = - \delta + D + \frac{1}{\I t }\ad(\varrho)$
        is real.
    \end{propositionlist}
    Suppose now in addition that $\algebra{A}$ is a $^*$-algebra over
    $\ring{C}$ with a $^*$-action of $\lie{g}$,
    i.e. \eqref{eq:StarAction}.
    \begin{propositionlist}
        \addtocounter{enumi}{3}
    \item The operator $L_{\algebra{A}}$ as well as the extended
        Fedosov derivation $\FedosovA$ are real.
    \item The Fedosov-Taylor series $\tau_{\algebra{A}}$ is real.
    \item The formal deformation $\star$ from
        Theorem~\ref{theorem:StarProduct} is a Hermitian deformation.
    \end{propositionlist}
\end{proposition}
When we apply this to the twist itself we first have to clarify which
$^*$-involution we take on the tensor algebra
$\Tensor^\bullet(\ueac{\lie{g}})$: by the universal property of the
tensor algebra, there is a unique way to extend the $^*$-involution of
$\ueac{\lie{g}}$ as a $^*$-involution. With respect to this
$^*$-involution we have $r^* = - r$ since $r$ is not only real as an
element of $\lie{g}_{\ring{C}} \tensor \lie{g}_{\ring{C}}$ but also
antisymmetric, causing an additional sign with respect to the
$^*$-involution of $\Tensor^\bullet(\ueac{\lie{g}})$. Analogously, we
have $s^* = s$ for the real and symmetric part of $\pi$.
\begin{corollary}
    \label{corollary:TwistHermitian}%
    The Fedosov twist $\mathcal{F}$ is Hermitian.
\end{corollary}
\begin{proof}
    Indeed, $1 \in \ueac{\lie{g}}$ is Hermitian and hence
    $(1 \star 1)^* = 1^* \star 1^* = 1 \star 1$.
\end{proof}

Up to now we have not yet used the fact that $\ring{R}$ is ordered but
only that we have a $^*$-involution. The ordering of $\ring{R}$ allows
to transfer concepts of positivity from $\ring{R}$ to every
$^*$-algebra over $\ring{C}$. Recall that a linear functional
$\omega\colon \algebra{A} \longrightarrow \ring{C}$ is called
\emph{positive} if
\begin{equation}
    \label{eq:omegaPositive}
    \omega(a^*a) \ge 0
\end{equation}
for all $a \in \algebra{A}$. This allows to define an algebra element
$a \in \algebra{A}$ to be \emph{positive} if $\omega(a) \ge 0$ for all
positive $\omega$. Note that the positive elements denoted by
$\algebra{A}^+$, form a convex cone in $\algebra{A}$ and
$a \in \algebra{A}^+$ implies $b^*ab \in \algebra{A}^+$ for all
$b \in \algebra{A}$. Moreover, elements of the form $a = b^*b$ are
clearly positive: their convex combinations are denoted by
$\algebra{A}^{++}$ and called \emph{algebraically positive}. More
details on these notions of positivity can be found in
\cite{bursztyn.waldmann:2005b, bursztyn.waldmann:2001a,
  waldmann:2005b}.

Since with $\ring{R}$ also $\ring{R}[[ t ]]$ is ordered, one can
compare the positive elements of $\algebra{A}$ and the ones of
$(\algebra{A}[[ t ]], \star)$, where $\star$ is a Hermitian
deformation. The first trivial observation is that for a positive
linear functional
$\boldsymbol{\omega} = \omega_0 +  t  \omega_1 + \cdots$ of the
deformed algebra, i.e. $\boldsymbol{\omega}(a^* \star a) \ge 0$ for
all $a \in \algebra{A}[[ t ]]$ the classical limit $\omega_0$ of
$\boldsymbol{\omega}$ is a positive functional of the undeformed
algebra. The converse needs not to be true: one has examples where a
positive $\omega_0$ is not directly positive for the deformed
algebras, i.e. one needs higher order corrections, and one has
examples where one simply can not find such higher order corrections
at all, see \cite{bursztyn.waldmann:2005a,
  bursztyn.waldmann:2000a}. One calls the deformation $\star$ a
\emph{positive deformation} if every positive linear functional
$\omega_0$ of the undeformed algebra $\algebra{A}$ can be deformed
into a positive functional
$\boldsymbol{\omega} = \omega_0 +  t \omega_1 + \cdots$ of the
deformed algebra $(\algebra{A}[[ t ]], \star)$. Moreover, since also
$\Mat_n(\algebra{A})$ is a $^*$-algebra in a natural way we call
$\star$ a \emph{completely positive deformation} if for all $n$ the
canonical extension of $\star$ to $\Mat_n(\algebra{A})[[ t ]]$ is a
positive deformation of $\Mat_n(\algebra{A})$, see
\cite{bursztyn.waldmann:2005a}. Finally, if no higher order
corrections are needed, then $\star$ is called a \emph{strongly
  positive deformation}, see \cite[Def.~4.1]{bursztyn.waldmann:2000a}

In a next step we want to use a Kähler structure for $\lie{g}$. In
general, this will not exist so we have to require it explicitly. In
detail, we want to be able to find a basis
$e_1, \ldots, e_n, f_1, \ldots, f_n \in \lie{g}$ with the property
that the $r$-matrix decomposes into
\begin{equation}
    \label{eq:ABpartOfr}
    (e^k \tensor f^\ell)(r)
    =
    A^{k\ell}
    =
    - (f^\ell \tensor e^k)(r)
    \quad
    \textrm{and}
    \quad
    (e^k \tensor e^\ell)(r)
    =
    B^{k\ell}
    =
    - (f^k \tensor f^\ell)(r)
\end{equation}
with a symmetric matrix $A = A^\Trans \in \Mat_n(\ring{R})$ and an
antisymmetric matrix $B = - B^\Trans \in \Mat_n(\ring{R})$. We set
\begin{equation}
    \label{eq:DefinitionOfs}
    s
    =
    A^{k\ell} (e_k \tensor e_\ell + f_k \tensor f_\ell)
    +
    B^{k\ell} e_k \tensor f_\ell
    +
    B^{k\ell} f_\ell \tensor e_k.
\end{equation}
The requirement of being \emph{Kähler} is now that first we find a
symplectic covariant derivative $\nabla$ with $\nabla s = 0$. Second,
we require the symmetric two-tensor $s$ to be positive in the sense
that for all $x \in \lie{g}^*$ we have $(x \tensor x)(s) \ge 0$.  In
this case we call $s$ (and the compatible $\nabla$) a Kähler structure
for $r$.  We have chosen this more coordinate-based formulation over
the invariant one since in the case of an ordered ring $\ring{R}$
instead of the reals $\mathbb{R}$ it is more convenient to start
directly with the nice basis we need later on.

As usual we consider now $\lie{g}_{\ring{C}}$ with the vectors
\begin{equation}
    \label{eq:ZkccZell}
    Z_k
    =
    \frac{1}{2}(e_k - \I f_\ell)
    \quad
    \textrm{and}
    \quad
    \cc{Z}_\ell
    =
    \frac{1}{2}(e_k + \I f_\ell)
\end{equation}
which together constitute a basis of the complexified Lie
algebra. Finally, we have the complex matrix
\begin{equation}
    \label{eq:complexMatrixg}
    g = A + \I B \in \Mat_n(\ring{C}),
\end{equation}
which satisfies now the positivity requirement
\begin{equation}
    \label{eq:gPositive}
    \cc{z_k} g^{k\ell} z_\ell \ge 0
\end{equation}
for all $z_1, \ldots, z_n \in \ring{C}$.  If our ring $\ring{R}$ has
sufficiently many inverses and square roots, one can even find a basis
$e_1, \ldots, e_n, f_1, \ldots, f_n$ such that $g$ becomes the unit
matrix. However, since we want to stay with an arbitrary ordered ring
$\ring{R}$ we do not assume this.

We use now $\pi = r + \I s$ to obtain a fiberwise Hermitian product
$\circwick$, called the fiberwise Wick product. Important is now the
following explicit form of $\circwick$, which is a routine
verification:
\begin{lemma}
    \label{lemma:FiberwiseWick}%
    For the fiberwise Wick product $\circwick$ build out of $\pi = r + \I
    s$ with a Kähler structure $s$ one has
    \begin{equation}
        \label{eq:FiberwiseWick}
        a \circwick b
        =
        \mu \circ
        \E^{2 t  g^{k\ell} \inss(Z_k) \tensor \inss(\cc{Z}_\ell)}
        (a \tensor b),
    \end{equation}
    where $g$ is the matrix from \eqref{eq:complexMatrixg}.
\end{lemma}

The first important observation is that the scalar matrix $g$ can be
viewed as element of $\Mat_n(\algebra{A})$ for any unital
$^*$-algebra. Then we have the following positivity property:
\begin{lemma}
    \label{lemma:gIsReallyPositive}%
    Let $\algebra{A}$ be a unital $^*$-algebra over $\ring{C}$. Then
    for all $m \in \mathbb{N}$ and for all
    $a_{k_1 \ldots k_m} \in \algebra{A}$ with
    $k_1, \ldots, k_m = 1, \ldots, n$ we have
    \begin{equation}
        \label{eq:TheTruePositivityNeeded}
        \sum_{k_1, \ell_1, \ldots, k_m, \ell_m = 1}^n
        g^{k_1\ell_1} \cdots g^{k_m\ell_m}
        a_{k_1 \ldots k_m}^* a_{\ell_1 \ldots \ell_m}
        \in
        \algebra{A}^+.
    \end{equation}
\end{lemma}
\begin{proof}
    First we note that
    $g^{\tensor m} = g \tensor \cdots \tensor g \in \Mat_n(\ring{C})
    \tensor \cdots \tensor \Mat_n(\ring{C}) = \Mat_{n^m}(\ring{C})$
    still satisfies the positivity property
    \[
    \sum_{k_1, \ell_1, \ldots, k_m, \ell_m = 1}^n
    g^{k_1\ell_1} \cdots g^{k_m\ell_m}
    \cc{z^{(1)}}_{k_1} \cdots \cc{z^{(m)}}_{k_m}
    z^{(1)}_{\ell_1} \cdots z^{(m)}_{\ell_m}
    \ge
    0
    \]
    for all $z^{(1)}, \ldots, z^{(m)} \in \ring{C}^n$ as the left hand
    side clearly factorizes into $m$ copies of the left hand side of
    \eqref{eq:gPositive}. Hence
    $g^{\tensor m} \in \Mat_{n^m}(\ring{C})$ is a positive
    element. For a given positive linear functional
    $\omega\colon \algebra{A} \longrightarrow \ring{C}$ and
    $b_1, \ldots, b_N \in \algebra{A}$ we consider the matrix
    $(\omega(b_i^*b_j)) \in \Mat_N(\ring{C})$.  We claim that this
    matrix is positive, too. Indeed, with the criterion from
    \cite[App.~A]{bursztyn.waldmann:2001a} we have for all
    $z_1, \ldots, z_N \in \ring{C}$
    \[
    \sum_{i,j = 1}^N \cc{z}_i \omega(b_i^*b_j) z_j
    =
    \omega\left(
        \left(
            \sum_{i=1}^N z_i b_i
        \right)^*
        \left(
            \sum_{j=1}^N z_j b_j
        \right)
    \right)
    \ge 0
    \]
    and hence $(\omega(b_i^*b_j))$ is positive.  Putting these
    statements together we see that for every positive linear
    functional $\omega\colon \algebra{A} \longrightarrow \ring{C}$ we
    have for the matrix
    $\Omega = (\omega(a_{k_1 \ldots k_m}^* a_{\ell_1 \ldots \ell_m}))
    \in \Mat_{n^m}(\ring{C})$
    \begin{align*}
        \omega\left(
            \sum_{k_1, \ell_1, \ldots, k_m, \ell_m = 1}^n
            g^{k_1\ell_1} \cdots g^{k_m\ell_m}
            a_{k_1 \ldots k_m}^* a_{\ell_1 \ldots \ell_m}
        \right)
        &=
        \sum_{k_1, \ell_1, \ldots, k_m, \ell_m = 1}^n
        g^{k_1\ell_1} \cdots g^{k_m\ell_m}
        \omega\left(
            a_{k_1 \ldots k_m}^* a_{\ell_1 \ldots \ell_m}
        \right) \\
        &=
        \tr(g^{\tensor m} \Omega)
        \ge
        0,
    \end{align*}
    since the trace of the product of two positive matrices is
    positive by \cite[App.~A]{bursztyn.waldmann:2001a}. Note that for
    a \emph{ring} $\ring{R}$ one has to use this slightly more
    complicated argumentation: for a field one could use the
    diagonalization of $g$ instead. By definition of $\algebra{A}^+$,
    this shows the positivity of \eqref{eq:TheTruePositivityNeeded}.
\end{proof}
\begin{remark}
    \label{remark:Diagonalg}%
    Suppose that in addition $g = \diag(\lambda_1, \ldots, \lambda_n)$
    is diagonal with positive $\lambda_1, \ldots, \lambda_n > 0$. In
    this case one can directly see that the left hand side of
    \eqref{eq:TheTruePositivityNeeded} is a convex combination of
    squares and hence in $\algebra{A}^{++}$. This situation can often
    be achieved, e.g. for $\ring{R} = \mathbb{R}$.
\end{remark}
We come now to the main theorem of this section: unlike the Weyl-type
deformation, using the fiberwise Wick product yields a positive
deformation in a universal way:
\begin{theorem}
    \label{theorem:WickIsPositive}%
    Let $\algebra{A}$ be a unital $^*$-algebra over
    $\ring{C} = \ring{R}(\I)$ with a $^*$-action of $\lie{g}$ and let
    $\Omega = \Omega^* \in \Anti^2 \lie{g}_{\ring{C}}^*$ be a formal
    series of Hermitian $\delta_\CE$-closed two-forms. Moreover, let
    $s$ be a Kähler structure for the non-degenerate $r$-matrix
    $r \in \lie{g}$ and consider the fiberwise Wick product
    $\circwick$ yielding the Hermitian deformation $\starwick$ as in
    Proposition~\ref{proposition:HermitianFedosov}.
    \begin{theoremlist}
    \item \label{item:astara} For all $a \in \algebra{A}$ we have
        \begin{equation}
            \label{eq:astara}
            a^* \starwick a
            =
            \sum_{m=0}^\infty \frac{(2 t )^m}{m!}
            \sum_{k_1, \ldots, k_m, \ell_1, \ldots, \ell_m = 1}^n
            g^{k_1\ell_1} \cdots g^{k_m\ell_m}
            a_{k_1\ldots k_m}^* a_{\ell_1 \ldots \ell_m},
        \end{equation}
        where
        $a_{k_1\ldots k_m} = \sigma\left( \inss(\cc{Z}_{k_1}) \cdots
            \inss(\cc{Z}_{k_m}) \tauwick(a) \right)$.
    \item \label{item:starwickPositive} The deformation $\starwick$ is
        strongly positive.
    \end{theoremlist}
\end{theorem}
\begin{proof}
    From Lemma~\ref{lemma:FiberwiseWick} we immediately obtain
    \eqref{eq:astara}. Now let
    $\omega\colon \algebra{A} \longrightarrow \ring{C}$ be
    positive. Then also the $\ring{C}[[ t ]]$-linear extension
    $\omega\colon \algebra{A}[[ t ]] \longrightarrow
    \ring{C}[[ t ]]$
    is positive with respect to the undeformed product: this is a
    simple consequence of the Cauchy-Schwarz inequality for
    $\omega$. Then we apply Lemma~\ref{lemma:gIsReallyPositive} to
    conclude that $\omega(a^* \star a) \ge 0$.
\end{proof}
\begin{corollary}
    \label{corollary:PositiveTwist}%
    The Wick-type twist $\twist{F}_{\mathrm{\scriptscriptstyle Wick}}$
    in the Kähler situation is a convex series of positive elements.
\end{corollary}
\begin{remark}[Positive twist]
    \label{remark:HermitianPositiveWhatever}%
    Note that already for a Hermitian deformation, the twist
    $\twist{F} = 1 \star 1 = 1^* \star 1$ constructed as above is a
    \emph{positive} element of the deformed algebra
    $\Tensor^\bullet(\ueac{\lie{g}})[[ t ]]$. However, this seems to
    be not yet very significant: it is the statement of
    Corollary~\ref{corollary:PositiveTwist} and
    Theorem~\ref{theorem:WickIsPositive} which gives the additional
    and important feature of the corresponding universal deformation
    formula.
\end{remark}

%
%

\appendix

%
%

\section{Hochschild-Kostant-Rosenberg theorem}
\label{sec:HKR}

Let us define the map
\begin{align}
    \label{eq:HKRDiff}
    \partial\colon \uea{\lie{g}}\ni \xi
    \; \mapsto \;
    \xi \tensor 1 + 1 \tensor \xi - \Delta(\xi)
    \in \uea{\lie{g}}^{\tensor 2},
\end{align}
and extend it as a graded derivation of degree $+1$ of the tensor
product to $\Tensor^\bullet(\uea{\lie{g}})$.  We recall that the map
$\partial\colon \Tensor^\bullet(\uea{\lie
  g})\to\Tensor^\bullet(\uea{\lie{g}})$
is a differential. Its cohomology is described as follows:
\begin{theorem}[Hochschild-Kostant-Rosenberg]
    \label{thm:HKR}%
    Let $C\in\Tensor^p(\uea{\lie{g}})$ such that $\partial C = 0$.
    Then there is a $X\in\Anti^k\lie{g}$ and a
    $S\in \Tensor^{p-1}(\uea{\lie g})$ with
    \begin{equation}
        \label{eq:HKR}
        C
        =
        X + \partial S
    \end{equation}
    with $X = \mathrm{Alt}(C)$.
\end{theorem}
We do not prove the above Theorem in full generality, since we need
only the case $p = 2$. In this case the proof consists of the
following two lemmas:
\begin{lemma}
    \label{Lem:skew-exact}%
    Let $C\in\Tensor^2(\uea{\lie{g}})$ with $\partial C =0$.
    \begin{enumerate}
    \item One has $\partial\mathtt{T}(C) = 0$
    \item The antisymmetric part satisfies
        $C - \mathtt{T}(C) \in \lie{g} \wedge \lie{g}\subseteq
        \Tensor^2(\uea{\lie{g}})$
    \end{enumerate}
\end{lemma}
\begin{proof}
    We have
    \begin{equation*}
        \partial C
        =
        0 \
        \Leftrightarrow \ C\tensor 1 + (\Delta\tensor \id)(C)
        =
        1\tensor C + (\id\tensor\Delta)(C).
    \end{equation*}
    Thus, we get
    \begin{align*}
        \mathtt{T}(C)\tensor 1
        &=
        (\mathtt{T}\tensor\id)(C\tensor 1)\\
        &=
        (\mathtt{T}\tensor\id)
        (1\tensor C + (\id\tensor\Delta)(C) - (\Delta\tensor \id)(C))
        \\
        &=
        C_{13}
        +
        (\mathtt{T}\tensor\id)(\id\tensor\Delta)(C)
        -
        (\Delta\tensor \id)(C).
    \end{align*}
    Now we apply the cyclic permutation to this equation and get
    \begin{equation*}
        1\tensor \mathtt{T}(C)
        =
        \mathtt{T}(C)\tensor 1
        +
        (\Delta\tensor \id)(\mathtt{T}(C))
        -
        (\id\tensor\Delta)(\mathtt{T}(C)),
    \end{equation*}
    which is equivalent to $\partial\mathtt{T}(C) = 0$.  Since
    $\partial$ is linear, we get $\partial(T-\mathtt{T}(C)) = 0$ and
    denote by $A = T - \mathtt{T}(C)$, which is now skew-symmetric.
    We define $Q = (\Delta \tensor\id)A - A_{23} - A_{13}$ and get
    with the fact that $A$ is $\partial$-closed that
    $Q = -\mathtt{Alt}(Q)$. Therefore we have
    $Q = \mathtt{Alt}^3 Q =(-1)^3 Q = -Q$ and we can conclude $Q =
    0$.
    Thus, $A$ has to be primitive in the first argument and with the
    skew-symmetry we get the same statement for the second argument.
\end{proof}
\begin{lemma}
    Let $C\in\Tensor^2(\uea{\lie{g}})$ with $\partial C = 0$. Then
    there exists a $S\in\uea{\lie{g}}$ and a
    $X\in\lie{g}\wedge \lie{g}$, such that
    \begin{equation}
        \label{eq:HKRInTwo}
        C
        =
        X + \partial S,
    \end{equation}
    where $X = \frac{1}{2}(C - \mathtt{T}(C))$.
\end{lemma}
\begin{proof}
    It is clear from Lemma~\ref{Lem:skew-exact}, that $X$ is
    well-defined and we have to prove that symmetric $C$ are
    $\partial$-exact. So we assume that
    $C \in \Tensor^2(\uea{\lie{g}})$ is $\partial$-closed and
    symmetric. Let $k$ be the highest order appearing in $C$ and
    assume the claim is true for all $r < k$ (in the sense of the
    filtration of
    $\uea{\lie{g}} = \bigcup_{n\in \mathbb{N}_0}\uea{\lie{g}}_n$).
    The we can write for a given basis
    $\{e_i\}_{i \in \{1, \ldots, n\}}$
    \begin{equation*}
        C
        =
        \sum _{|\boldsymbol{i}|= k}
        e_{\boldsymbol{i}} \tensor D^{\boldsymbol{i}}
        +
        l.o.t. .
    \end{equation*}
    We mean lower order terms with respect to the filtration in the
    first tensor degree and $\boldsymbol{i}$ are multiindices, such
    that $e_{\boldsymbol{i}} = e_{i_1} \cdots e_{i_k}$. We can assume
    that $D_{\boldsymbol{i}}$ is symmetric in the multiindex, because
    we can compensate non-symmetricy by lower order terms.  Since
    $\partial (\uea{\lie{g}}_m)\subseteq
    \uea{\lie{g}}_{m-1}\tensor\uea{\lie{g}}_{m-1}$,
    we see that $\partial C = 0$ implies that
    $\partial D^{\boldsymbol{i}} = 0$, which is equivalent to
    $ D^{\boldsymbol{i}}\in \lie{g}$.  Therefore, we can write
    \begin{equation*}
        C
        =
        \sum _{|\boldsymbol{i}|= k}
        D^{\boldsymbol{i},j} e_{\boldsymbol{i}}\tensor e_j
        +
        H,
    \end{equation*}
    where $H\in\uea{\lie{g}}_{k-1} \tensor \uea{\lie{g}}$ is now of
    order strictly less then $k$ in the first argument. Now we expand
    $H = \sum _{|\boldsymbol{i}_1|,|\boldsymbol{i}_{2}|\leq k-1}
    H_{\boldsymbol{i}_1,\boldsymbol{i}_2} e_{\boldsymbol{i}_1} \tensor
    e_{\boldsymbol{i}_{2}}$ and see, by using
    \begin{align*}
        0
        &=
        \partial C \\
        &=
        \sum _{|\boldsymbol{i}|= k}
        D^{\boldsymbol{i},j} \partial(e_{\boldsymbol{i}}) \tensor e_j
        +
        \partial H \\
        &=
        - D^{i_1,\dots,i_k,j} \sum_r
        e_{i_1}
        \cdots \widehat{e_{i_r}} \cdots
        e_{i_k} \tensor e_{i_r} \tensor e_j
        +
        \partial H + l.o.t.,
    \end{align*}
    that $H$ has to be of the form
    \begin{equation*}
        H
        =
        \sum _{|\boldsymbol{i}_1|={k-1}, |\boldsymbol{i}_{2}|= 2}
        H_{\boldsymbol{i}_1,\boldsymbol{i}_2}
        e_{\boldsymbol{i}_1}\tensor e_{\boldsymbol{i}_{2}}
        +
        l.o.t.,
    \end{equation*}
    and hence
    \begin{equation*}
        \partial H
        =
        \sum _{|\boldsymbol{i}_1|={k-1},j_1,j_2} H_{\boldsymbol{i}_1,j_1,j_2}
        e_{\boldsymbol{i}_1}\tensor e_{j_1}\tensor e_{j_2}+l.o.t..
    \end{equation*}
    This implies, that $D^{i_1, \ldots, i_k, j}$ is symmetric in all
    indices, since $\partial C = 0$ and
    $H_{\boldsymbol{i}_1,j_1,j_2} =
    H_{\boldsymbol{i}_1,j_2,j_1}$. Thus for
    \begin{equation*}
        G
        =
        \frac{1}{k+1} D^{i_1,\dots,i_{k+1}}e_{i_1}\cdots e_{i_{k+1}}
    \end{equation*}
    we have
    \begin{equation*}
        \partial G
        =
        - \sum _{|\boldsymbol{i}|= k}
        D^{\boldsymbol{i},j}
        (e_{\boldsymbol{i}}\tensor e_j+e_j\tensor e_{\boldsymbol{i}})
        +
        l.o.t..
    \end{equation*}
    Note that here the lower order terms are meant in both tensor
    arguments. Using the symmetry of $C$, we obtain
    \begin{equation*}
        C
        =
        \sum _{|\boldsymbol{i}|= k}
        D^{\boldsymbol{i},j}
        (e_{\boldsymbol{i}}\tensor e_j+e_j\tensor e_{\boldsymbol{i}})
        +
        l.o.t.,
    \end{equation*}
    again the lower order terms are in both tensor factors. Thus,
    \begin{align*}
        C + \partial G
        \in
        \uea{\lie{g}}_{k-1}\tensor\uea{\lie{g}}_{k-1}.
    \end{align*}
    This implies the Lemma, because for $k = 0$ the statement is
    trivial.
\end{proof}
\begin{corollary}
    Let $C \in \Tensor^2(\uea{\lie{g}})$ with $\partial C = 0$ and
    $(\epsilon\tensor\id)C = (\id\tensor\epsilon)C = 0$ , then we can
    find $S \in \uea{\lie{g}}$ and $X \in \Anti^2\lie g$, such that
    $C = X + \partial S$ with $\epsilon(S) = 0$.
\end{corollary}
\begin{proof}
    The statement is clear from the construction of
    Lemma~\ref{Lem:skew-exact}.
\end{proof}

%
%

\section{Technical Lemmas}

In this section we prove several technical results, necessary for the
proofs is Section~\ref{sec:Classification}.
\begin{lemma}
    \label{lem:ExactnessDifference}%
    Let
    $\twist{F}, \twist{F}' \in (\uea{\lie{g}}\tensor \uea{\lie
      g})[[t]]$ be two twists coinciding up to order $k$. Then
    \begin{equation}
        \label{eq:DiffTwistsClosed}
        \partial (F_{k+1} - F'_{k+1}) = 0.
    \end{equation}
\end{lemma}
\begin{proof}
    We have
    \begin{align*}
        \partial(F_{k+1})
        &=
        1\tensor F_{k+1}
        - F_{k+1}\tensor 1
        + (\id\tensor\Delta)(F_{k+1})
        - (\Delta\tensor \id)(F_{k+1})\\
        &=
        \sum_{i=0}^{k+1}
        (1\tensor F_{i})(\id\tensor\Delta)(F_{k+1-i})
        -
        \sum_{i=1}^k
        (1\tensor F_{i})(\id\tensor\Delta)(F_{k+1-i})\\
        &\quad+
        \sum_{i=1}^k
        (F_{i}\tensor 1)(\Delta\tensor \id)(F_{k+1-i})
        -
        \sum_{i=0}^{k+1}
        (F_{i}\tensor 1)(\Delta\tensor \id)(F_{k+1-i})
        \\
        &=
        -
        \sum_{i=1}^k
        (1\tensor F_{i})(\id\tensor\Delta)(F_{k+1-i})
        +
        \sum_{i=1}^k
        (F_{i}\tensor 1)(\Delta\tensor \id)(F_{k+1-i})\\
        &=
        -\sum_{i=1}^k
        (1\tensor F'_{i})(\id\tensor\Delta)(F'_{k+1-i})
        +
        \sum_{i=1}^k
        (F'_{i}\tensor 1)(\Delta\tensor \id)(F'_{k+1-i})\\
        &=
        \partial(F'_{k+1}).
    \end{align*}
\end{proof}
\begin{lemma}
    \label{lem:EquivalenceOrder}%
    Let
    $\twist{F}, \twist{F}' \in (\uea{\lie{g}}\tensor \uea{\lie
      g})[[t]]$ be two twists coinciding up to order $k$, such that
    \begin{equation}
        \label{eq:DiffTwistsTkplusOne}
        F_{k+1}-F'_{k+1} = \partial T_{k+1}.
    \end{equation}
    Then they are equivalent up to order $k+1$.
\end{lemma}
\begin{proof}
    Consider
    $\exp(t^{k+1} T_{k+1}) = 1+t^{k+1} T_{k+1} +
    \mathcal{O}(t^{k+2})$. Then we have
    \begin{align*}
        (\Delta(\exp(t^{k+1} T_{k+1}))\twist{F})_i
        =
        \left(
            \twist{F}'
            \left(
                \exp(t^{k+1} T_{k+1})\tensor\exp(t^{k+1} T_{k+1})
            \right)
        \right)_i
    \end{align*}
    for any $i \leq k+1$.  Note that, because
    $(\epsilon\tensor\id)(F_{k+1} - F'_{k+1}) =
    (\id\tensor\epsilon)(F_{k+1}-F'_{k+1}) = 0$,
    we can choose $T_{k+1}$, such that $\epsilon (T_{k+1}) = 0$ and
    therefore $\epsilon(\exp(t^{k+1} T_{k+1})) = 1$.
\end{proof}
\begin{lemma}
  \label{lem:EquivalenceOfTwists}%
  Let
  $\twist{F}, \twist{F}' \in (\uea{\lie{g}}\tensor
  \uea{\lie{g}})[[t]]$
  be two equivalent twists coinciding up to order $k$. Then there
  exists a
  $T = 1 + t^k T_k + \mathcal{O}(t^{k+1}) \in \uea{\lie{g}}[[t]]$,
  such that
  \begin{equation}
      \label{eq:DeltaTTwistFprime}
      \Delta(T)\twist{F}'
      =
      \twist{F}(T\tensor T).
  \end{equation}
\end{lemma}
\begin{proof}
    Since the twists $\twist{F}$ and $\twist{F}'$ are equivalent,
    there is a
    $\tilde{T} = 1 + t^\ell \tilde{T}_\ell + \mathcal{O}(t^{\ell+1})$,
    such that
    \begin{equation*}
        \Delta(\tilde{T})F'
        =
        F(\tilde{T} \tensor \tilde{T}).
    \end{equation*}
    Let us consider $\ell \leq k$. The above equation at order $\ell$
    reads
    \begin{equation*}
        \Delta(\tilde{T}_\ell) + F'_\ell
        =
        F_\ell + \tilde{T}_\ell \tensor 1 + 1 \tensor \tilde{T}_\ell.
    \end{equation*}
    Therefore, since $\twist{F}$ and $\twist{F}'$ coincide up to order
    $k$ we have
    \begin{equation*}
        \Delta(\tilde{T}_\ell)
        =
        \tilde{T}_\ell \tensor 1 + 1 \tensor \tilde{T}_\ell,
    \end{equation*}
    and we have $\tilde{T}_\ell \in \lie{g} \subseteq \uea{\lie{g}}$.
    For $\ell < k$ we get at order $\ell+1$
    \begin{equation*}
        \Delta(\tilde{T}_{\ell+1})
        +
        \Delta(\tilde{T}_\ell) F'_1 + F'_{\ell+1}
        =
        F_{\ell+1}
        +
        F_1 (\tilde{T}_\ell \tensor 1 + 1 \tensor \tilde{T}_\ell)
        +
        \tilde{T}_{\ell+1} \tensor 1 + 1 \tensor \tilde{T}_{\ell+1}.
    \end{equation*}
    The skew-symmetrization of the above equation gives
    \begin{equation*}
        (\tilde{T}_\ell \tensor 1 + 1 \tensor \tilde{T}_\ell) r
        =
        r (\tilde{T}_\ell \tensor 1 + 1 \tensor \tilde{T}_\ell).
    \end{equation*}
    An easy computation shows that this property is equivalent to
    $\delta_\CE \tilde{T}_\ell^\flat = 0$. Thus, we can define the map
    $S\colon \uea{\lie{g}} \longrightarrow \uea{\lie{g}}$ by defining
    it on primitive elements via
    \begin{equation*}
        \lie{g} \ni \xi
        \; \mapsto \;
        \tilde{T}_\ell^\flat(\xi) \cdot 1 \in \uea{\lie{g}}
    \end{equation*}
    and extend it as a derivation of the product of
    $\uea{\lie{g}}$. This map allows us to define an element
    \begin{equation*}
        A
        =
        \frac{1}{t}(\epsilon\circ S\tensor \id) [\twist{F}]
        =
        -\tilde{T}_\ell + \mathcal{O}(t),
    \end{equation*}
    which fulfills
    $\Delta(A)\twist{F} = \twist{F}(A \tensor 1 + 1 \tensor A)$ and
    $\epsilon(A)=0$.  Thus we get
    \begin{equation*}
        \exp(t^\ell A)\twist{F}
        =
        \twist{F}(\exp(t^\ell A)\tensor \exp(t^\ell A))
        \quad
        \textrm{as well as}
        \quad
 	\epsilon(\exp(t^\ell A)) = 1.
    \end{equation*}
    We define $T = \exp(t^\ell A) \tilde{T}$ and obtain
    $\Delta(T)\twist{F}' = \twist{F}(T\tensor T)$ and
    $T = 1 + t^{\ell+1} T_{\ell+1} + \mathcal{O}(t^{\ell+2})$.
    Repeating this method $k - \ell$ times, we get an equivalence
    starting at order $k$.
\end{proof}
\begin{lemma}
    \label{lem:SkewsymmetricTwistExact}%
    Let
    $\twist{F}, \twist{F}' \in (\uea{\lie{g}} \tensor
    \uea{\lie{g}})[[t]]$
    be two equivalent twists coinciding up to order $k$. Then there
    exists an element $\xi \in \lie{g}^*$, such that
    \begin{equation}
        \label{eq:DiffTwistsOrderkplusOne}
        ([F_{k+1} - F'_{k+1}])^\flat
        =
        \delta_\CE \xi.
    \end{equation}
\end{lemma}
\begin{proof}
    First, $[F_{k+1} - F'_{k+1}] \in \Anti^2 \lie{g}$, because of
    Theorem~\ref{thm:HKR} and since $\partial(F_{k+1}- F '_{k+1}) = 0$
    as in Lemma~\ref{lem:ExactnessDifference}.  From
    Lemma~\ref{lem:EquivalenceOfTwists} we know that we can find an
    element $T = 1 + t^k T_k + \mathcal{O}(t^{k+1})$ in
    $\uea{\lie{g}}$, such that
    $\Delta(T)\twist{F}' = \twist{F}(T\tensor T)$.  At order $k$ this
    reads
    \begin{equation*}
        \Delta(T_k) + F'_k
        =
        F_k + T_k \tensor 1 + 1 \tensor T_k,
    \end{equation*}
    which is equivalent to $T_k \in \lie{g}$, because $F'_k = F_k$. At
    order $k+1$, we can see that
    \begin{equation*}
        \Delta(T_{k+1}) + \Delta(T_k)F'_1 + F'_{k+1}
        =
        F_{k+1}
        +
        F_1(T_k\tensor 1 + 1 \tensor T_k)
        +
        T_{k+1} \tensor 1
        +
        1 \tensor T_{k+1}.
    \end{equation*}
    For the skew-symmetric part we have
    \begin{equation*}
        [F_{k+1} - F'_{k+1}]
        =
        (T_k\tensor 1 + 1\tensor T_k) r
        -
        r (T_k\tensor 1+1\tensor T_k)
        =
        [T_k \tensor 1 + 1 \tensor T_k, r],
    \end{equation*}
    which is equivalent to
    $([F_{k+1}-F'_{k+1}])^\flat = -\delta_\CE T_k^\flat$.
\end{proof}

%
%

{
  \footnotesize
  \renewcommand{\arraystretch}{0.5}

}

\end{document}